\documentclass[a4paper,12pt, reqno]{amsart}
\usepackage{amsfonts}
\usepackage{amsmath, array, bigints}
\usepackage{amssymb}
\usepackage{mathrsfs}
\usepackage[colorlinks]{hyperref}
 \usepackage{graphicx}
\usepackage{enumerate}

 \usepackage{esint}
 \usepackage[usenames]{color}

\usepackage{mathtools}
\makeatletter
\newcommand{\setword}[2]{%
  \phantomsection
  #1\def\@currentlabel{\unexpanded{#1}}\label{#2}%
}
\makeatother
\newtheorem{thm}{Theorem}[section]
\newtheorem{cor}[thm]{Corollary}
\newtheorem{lem}[thm]{Lemma}

\numberwithin{equation}{section}
\usepackage[english]{babel} 
\usepackage{blindtext}

\theoremstyle{definition}
\newtheorem{definition}[thm]{Definition}
\newtheorem{rem}[thm]{Remark}

\begin{document}

\allowdisplaybreaks 
\def\dist{{\operatorname{dist}}}
\def\loc{{\operatorname{loc}}}
\def\ess{{\operatorname{ess~sup}}}
\def\essi{{\operatorname{ess~inf}}}
\def\esslim{{\operatorname{ess~lim~inf}}}
\def\supp{{\operatorname{supp}}}
\renewcommand{\d}{\:\! \mathrm{d}}


 \title[Regularity of superposition operators]{Regularity of superposition operators of mixed fractional order}

 \author[Souvik Bhowmick, Sekhar Ghosh, Vishvesh Kumar, and R. Lakshmi]{Souvik Bhowmick, Sekhar Ghosh, Vishvesh Kumar and R. Lakshmi}

\address[Souvik Bhowmick]{Department of Mathematics, National Institute of Technology Calicut, Kozhikode, Kerala, India - 673601}
\email{souvikbhowmick2912@gmail.com / souvik\_p230197ma@nitc.ac.in}

\address[Sekhar Ghosh]{Department of Mathematics, National Institute of Technology Calicut, Kozhikode, Kerala, India - 673601}
\email{sekharghosh1234@gmail.com / sekharghosh@nitc.ac.in}
\address[Vishvesh Kumar]{Department of Mathematics: Analysis, Logic and Discrete Mathematics, Ghent University, Ghent, Belgium}
\email{vishveshmishra@gmail.com / vishvesh.kumar@ugent.be}
\address[R. Lakshmi]{Department of Mathematics, National Institute of Technology Calicut, Kozhikode, Kerala, India - 673601}
\email{lakshmir1248@gmail.com / lakshmi\_p220223ma@nitc.ac.in}
\date{}

\begin{abstract}
We extend the De Giorgi--Nash--Moser theory to superposition operators of mixed fractional operators. In particular, we investigate several regularity properties for this class of operators. We establish the Caccioppoli-type inequality with tail for weak subsolutions, local boundedness of weak subsolutions, local H\"older continuity of weak solutions, the weak Harnack inequality for weak supersolutions, and the lower semicontinuity of weak supersolutions. Furthermore, we prove the expansion of positivity, a preliminary Harnack inequality, and the upper semicontinuity of weak subsolutions.

Our results apply to both fixed-sign and sign-changing solutions involving mixed local--nonlocal superposition fractional operators. Notably, the results are new even in the classical linear case $p=2$, demonstrating the broader applicability of the techniques developed in this work.
\end{abstract}

\subjclass[2020]{35B65, 35D30, 35B45, 35R09, 35R11, 35M12}
\keywords{Regularity theory, Superposition operators, Caccioppoli inequality, Local boundedness, H\"older continuity, Weak Harnack inequality, Lower and upper semicontinuity.}

\maketitle

\section{Introduction}\label{int}
\subsection{Motivation and State of the Art}
In the present paper, we investigate the regularity theory of the weak solution of the following nonlinear superposition fractional operator
\begin{equation}\label{SP}
    A_{\mu, p} u:=\int_{[0,1]}(-\Delta)_{p}^{s} u \d \mu(s)=\int_{(0,1)}(-\Delta)_{p}^{s} u \d \mu(s)-\alpha \Delta_p u,
\end{equation} 
 where $\mu$ is a nonnegative, nontrivial finite Borel measure over $[0,1]$ with  $\alpha:=\mu(\{1\})\geq 0$ and $\mu(\{0\})= 0.$ The local $p$-Laplace operator is defined by 
 $\Delta_pu = \text{div}(|\nabla u|^{p-2} \nabla u)$ and the notation $(-\Delta)_p^s$ is conventionally assigned to the fractional $p$-Laplacian, defined for all $s \in (0,1)$ by
$$
(-\Delta)_{p}^{s} u(x):=2 C_{N,p,s} \lim _{\varepsilon \backslash 0} \int_{\mathbb{R}^{N} \backslash B_{\varepsilon}(x)} \frac{|u(x)-u(y)|^{p-2}(u(x)-u(y))}{|x-y|^{N+s p}} \d y.
$$
The positive normalizing constant $C_{N,p,s}$ is chosen in such a way that, it provides consistent limits as $s \nearrow 1$ and as $s \searrow 0$, namely
\begin{align*}
& \lim _{s \searrow 0}(-\Delta)_{p}^{s} u=(-\Delta)_{p}^{0} u:=|u|^{p-2}u, \\
& \lim _{s \nearrow 1}(-\Delta)_{p}^{s} u=(-\Delta)_{p}^{1} u:=-\Delta_{p} u=-\operatorname{div}\left(|\nabla u|^{p-2} \nabla u\right).
\end{align*}

The operator $A_{\mu, p}$ was introduced in \cite{DPSV}, which extended the construction of the linear case (i.e., $p=2$) discussed in \cite{DPSV2}. Particular cases for the operator in \eqref{SP} are (minus) the $p$-Laplacian (corresponding to the choice of $\mu$ being the Dirac measure concentrated at 1 ), the fractional $p$-Laplacian $(-\Delta)_p^{s}$ (corresponding to the choice of $\mu$ being the Dirac measure concentrated at some fractional power $s \in (0, 1)$ ), the ``mixed order operator" $-\Delta_p+\epsilon(-\Delta)_p^{s}$ (when $\mu$ is the sum of two Dirac measures $\delta_1+\epsilon\delta_s,\,s \in (0, 1)$ and $\epsilon\in (0,1]$ ) etc. 
{\it An intriguing and novel aspect of the operators studied in this work is their ability to simultaneously handle nonlinear operators and an infinite (potentially uncountable) family of fractional operators.} Beyond their theoretical significance, these operators are instrumental in fields like mathematical biology to model population dispersal. They are particularly effective at capturing diverse movement strategies, ranging from standard Gaussian diffusion to Lévy flights (see \cite{ DPSV25,DV21}).

There is an extensive body of literature on the regularity theory for the particular case of the nonlocal operator $A_{\mu,p}$, which is too vast to be exhaustively cited here. In particular, when $\mu$ is the Dirac measure concentrated at $1$, that is, $A_{\mu,p} = -\Delta_p$, we refer to the classical monograph \cite{MZ97}, the book \cite{Lind19}, and the references therein.

On the other hand, for the regularity theory of the nonlocal $p$-Laplacian, namely $A_{\mu,p} = (-\Delta)^s_p$, we refer to \cite{ DKP2014, DKP2016,Kass11} and references therein. More precisely, Di Castro, Kuusi, and Palatucci \cite{DKP2014} established a general Harnack inequality for minimizers of nonlocal, possibly degenerate, integro-differential operators, with the fractional $p$-Laplacian $(-\Delta)^s_p$ as a model case. This result extends the Harnack inequality in the linear setting ($p=2$), previously obtained by Kassmann \cite{Kass11}.

It is known that a scale-invariant Harnack inequality holds for globally nonnegative solutions of $(-\Delta)^s_p u = 0$. However, Kassmann \cite{Kass11} constructed a counterexample demonstrating that the positivity assumption is essential and cannot be weakened or removed, even in the linear case $p=2$. To address this issue, he introduced a revised formulation of the Harnack inequality that avoids the requirement of global positivity in $\mathbb{R}^n$. This is achieved by incorporating an additional term on the right-hand side, a natural {\it tail contribution}, which captures the inherent nonlocal effects of the fractional Laplacian. For the fractional $p$-Laplacian, Di Castro, Kuusi, and Palatucci \cite{DKP2014, DKP2016} established several fundamental regularity results, including a Harnack inequality, local boundedness, Hölder continuity, and a Caccioppoli estimate involving a tail term for the weak solutions. One of their most significant and influential contributions is the introduction of the {\it nonlocal tail} introduced in \cite{DKP2016}, which plays a crucial role in proving these results and in effectively handling the nonlocal nature of the fractional $p$-Laplacian. Particularly, in \cite{DKP2016}, the authors extend the De Giorgi-Nash-Moser theory (see \cite{DeG57, Mos61, Nash58}) to a class of fractional $p$-Laplace equations. They provide the existence of a unique minimizer to homogeneous equations and prove local regularity estimates for weak solutions.  For the related results in this direction, we cite \cite{Cozz17, DRSV22,DSV17, DK2020,Kass07, OS26} and references therein. 

We now turn to a specific case of the operator $A_{\mu,p}$ in which $\mu$ is the sum of two Dirac measures, $\delta_1 + \delta_s$ with $s \in (0,1)$. This corresponds to the mixed local and nonlocal operator
\begin{equation} \label{mixed}
-\Delta_p u + (-\Delta)_p^s u = 0.
\end{equation}

The first results on the regularity theory of \eqref{mixed} were obtained by Foondun \cite{Foo09}, who proved the Harnack inequality and local H\"older continuity for nonnegative solutions in the linear case $p=2$. In the nonlinear setting, Garain and Kinnunen \cite{GK2022} studied \eqref{mixed} and established local boundedness, H\"older continuity, and a Harnack inequality, as well as the semicontinuity of weak solutions. Their approach is purely analytic and is based on the De Giorgi–Nash–Moser theory. A variety of further aspects of this class of equations have been thoroughly investigated, including results on existence, uniqueness, and higher H\"older regularity by Garain and Lindgren \cite{GL23}, as well as questions concerning measure data by Byun and Song \cite{BS2023}. We also highlight a recent contribution by De Filippis and Mingione \cite{DeFG24}, who established results on maximal regularity and global H\"older continuity of minimizers for mixed local and nonlocal functionals with nonuniform growth. These findings were further complemented by Ding, Fang, and Zhang \cite{DYZ24}, who investigated local boundedness, local H\"older continuity, and the Harnack inequality in this setting.

Now, let us go back to the superposition operator $A_{\mu,p}$, which has been a subject of interest in studying several important problems in PDEs. The operator $A_{\mu,p}$ is a particular case of  the general superposition operator $B_{\upsilon, p}$ of nonlinear fractional operators, initially introduced by Dipierro \textit{et al.} in \cite{DPSV2} for $p=2$ and subsequently generalized in \cite{DPSV} for $1<p<\infty$, which is defined as follows:
\begin{equation}\label{superposition operator}
    B_{\upsilon, p} u:=\int_{[0,1]}(-\Delta)_{p}^{s} u \d \upsilon(s), 
\end{equation} 
where $\upsilon$ is a signed measure given by
\begin{equation*} 
    \upsilon:=\upsilon^{+}-\upsilon^{-}. 
\end{equation*} Moreover, the measures $\upsilon^{+}$ and $\upsilon^{-}$ are two nonnegative finite Borel measures on $[0,1]$ holds the following conditions: 
\begin{equation}\label{measure 1}
    \upsilon^{+}([{s'}, 1])>0,
\end{equation}
\begin{equation*}
    \upsilon^{-}|_{[{s'}, 1]}=0,
\end{equation*}
and
\begin{equation*}
\upsilon^{-}([0, {s'}]) \leq \kappa \upsilon^{+}([{s'}, 1]) ,
\end{equation*}
for some ${s'} \in(0,1]$ and $\kappa \geq 0.$ From \eqref{measure 1}, there exists another (fractional) exponent $s_{\sharp} \in[{s'}, 1]$ such that 
\begin{equation*}
\upsilon^{+}\left(\left[s_{\sharp}, 1\right]\right)>0. 
\end{equation*}

Note that $\upsilon$ is a signed measure; thus, an important example of the operators is that certain components have a ``wrong sign". There are two works by Biagi \textit{et al.} \cite{BDVV2026} and Perera and Sportelli \cite{PK2024}, which are related to the ``wrong sign." 
For the linear case of the superposition operator $B_{\upsilon, 2}$,   the initial study done by Dipierro \textit{et al.} \cite{DPSV2} established the existence of nontrivial solutions to \eqref{superposition operator} involving jump nonlinearities and critical growth. Subsequently, in \cite{DPSV25}, the authors proved the existence of a positive eigenvalue associated with \eqref{superposition operator} and analyzed a stationary logistic equation of Fisher–Kolmogorov–Petrovskii–Piskunov type involving this operator.  Dipierro \textit{et al.} \cite{DLSV2026} studied spectral properties of \eqref{superposition operator}. Moreover, they prove that the regularity result, $u\in W^{2,q}(\Omega)$, when $u\in H_0^1(\Omega)$ is a weak solution of the following mixed local-nonlocal superposition problem:
\begin{align*}
    \int_{[0,\Tilde{s}]}(-\Delta)^{s} u \d \upsilon(s)- \Delta u=&f(x) \text{ in  }\Omega, \nonumber \\
    u=&0 \text{ in  }\mathbb{R}^N\setminus \Omega,
\end{align*}
where $\Omega$ is open bounded subset of $\mathbb{R}^N$ with boundary class of $C^1$, $\Tilde{s}\in [0,\frac{1}{2})$, $f\in L^q(\Omega)$ with $q\in (1,\infty)$. 
More recently, in \cite{DLSV2025}, the authors showed that {\it the maximum principle fails for \eqref{superposition operator} when the measure $\upsilon$ changes sign}. However, they proved that the maximum principle holds for the operator \eqref{SP}. In \cite{ABB2025}, Afonso \textit{et al.} established the existence of weak solutions to an asymptotic problem associated with \eqref{superposition operator}, viewed as a perturbation of the corresponding eigenvalue problem. Furthermore, in \cite{BMS2025}, the existence of solutions involving critical exponents and subcritical growth was obtained. Very recently, Proietti Lippi and Sportelli \cite{PS2026} proved the existence of a ground state solution for a Choquard equation involving the superposition operator. Finally, in \cite{DJV2025}, Dipierro \textit{et al.} studied a general framework for superpositions of operators.

In their pioneering work, Dipierro \textit{et al.} \cite{DPSV} introduced the nonlinear superposition operators and established the existence of $m$ pairs of distinct nontrivial solutions to the problem:
\begin{align}\label{problem 3}
B_{\upsilon,p} u &= \lambda |u|^{q-2}u+|u|^{p^*_{s_\sharp}-2}u ~\text{in}~\Omega,  \nonumber\\
 u & =0 ~\text{in}~ \mathbb{R}^N \setminus \Omega,
 \end{align}
 for the case $q=p$. Recently, the middle two authors, together with Aikyn and Ruzhansky \cite{AGKR2025}, studied a Brezis--Nirenberg type problem involving superposition operators. In \cite{AGKR2025i}, they further analyzed spectral properties and shape optimization of \eqref{superposition operator} via variational methods. They also established the weak and strong maximum principles, along with logarithmic estimates for \eqref{SP}, by introducing the notion of a nonlocal tail. These results are obtained for \eqref{SP} since the maximum principle holds if and only if $\upsilon$ has constant sign. Very recently, the first three authors \cite{BGK2026a} proved the existence of nontrivial solutions to \eqref{problem 3} for $1<q<p<p^*_{s_\sharp}$ using variational methods, truncation techniques, and genus theory. In \cite{BGK2026}, they further established the existence of infinitely many solutions to \eqref{superposition operator} under various superlinear growth conditions using the Palais--Smale and Cerami conditions, along with the Fountain Theorem. 
We also refer to \cite{ DLSV2024,DPSV1} for the general theory of $(s,p)$ superposition operators and related optimal embedding results in fractional Sobolev spaces.

 Despite significant progress on nonlinear problems involving the superposition operator $A_{\mu,p}$, as well as more general superposition operators, the regularity theory for $A_{\mu,p}$ remains largely unexplored. The primary aim of this paper is therefore to address this gap and develop a regularity theory for $A_{\mu,p}$. We emphasize that our results are new even in the linear case $p=2$. The analysis is particularly challenging due to the mixed local-nonlocal structure of the operator, which simultaneously incorporates an infinite (possibly uncountable) family of fractional operators of different orders. Consequently, one must handle not only the intrinsic nonlocal interactions and the interplay between distinct differentiability scales, but also the additional difficulties arising from the nonlinear nature of the operator. One of the fundamental tools in our analysis is the {\it nonlocal superposition tail} introduced in \cite{AGKR2025i}, naturally associated with the operator $A_{\mu,p}$. For $u \in W^{s,p}_{\mu}(\mathbb{R}^N)$ (see Section \ref{pre}), the nonlocal superposition tail in the ball $B_r(x_0)$ is defined by
\begin{equation} \label{tailinro}
\operatorname{Tail}(u;x_0,r)
=
\left[
\int_{(0,1)}
C_{N,s,p}\, r^{sp}
\left(
\int_{\mathbb{R}^N \setminus B_r(x_0)}
\frac{|u(y)|^{p-1}}{|y-x_0|^{N+sp}}
\,dy
\right)
d\mu(s)
\right]^{\frac{1}{p-1}}.
\end{equation}

We point out that the nonlocal superposition tail was originally introduced in \cite{AGKR2025i} without the normalizing constant $C_{N,s,p}$. However, the inclusion of this normalization serves our purposes well and plays an important role in several parts of the analysis.

\subsection{Assumptions, Main Results and Discussions} Throughout the paper, we assume the following condition on the measure defining the superposition operators; any additional assumptions will be stated explicitly when needed.  We assume that the measure $\mu$ defining the superposition operator $A_{\mu,p}$
given by 
$$A_{\mu, p} u:=\int_{[0,1]}(-\Delta)_{p}^{s} u \d \mu(s)=\int_{(0,1)}(-\Delta)_{p}^{s} u \d \mu(s)-\alpha \Delta_p u$$


where $\mu$ is a nonnegative and finite Borel measure $[0,1],$ and nontrivial on $(0, 1)$ with $\mu(\{0\}) = 0.$ The condition $\mu(\{0\}) = 0$ is imposed to preserve the scaling properties of the operator $A_{\mu,p}$.  We set 
\[
\alpha := \mu(\{1\}) \ge 0, \qquad \Sigma := \operatorname{supp}\{\mu\}.
\]

Since $\mu$ is nontrivial, $\Sigma$ is nonempty and there exists a exponent $s'\in (0,1]$ such that 
\begin{equation}\label{m4}
\mu\left(\left [s', 1\right]\right)>0. 
\end{equation} 
Moreover, using \eqref{m4}, there exists another fractional exponent $\bar{s_\sharp}\in [{s'},1]$ such that 
\begin{equation*}
\mu\left(\left [\bar{s_\sharp}, 1\right]\right)>0.
\end{equation*} 

Therefore, we define
\begin{equation}\label{m5}
    s_\sharp =\begin{cases}
        & \bar{s_\sharp} ~\text{ if } \alpha=0, \\
        & 1 ~ \text{ if } \alpha> 0.
    \end{cases}
\end{equation}

Subsequently, we will see that $s_\sharp$ also serves as a critical exponent, where the fractional critical exponent $p_{s_{\sharp}}^{*}$ is defined as $ p_{s_{\sharp}}^{*}:=\frac{N p}{N-s_{\sharp} p}$ related to the fractional exponent $s_{\sharp}$ for which \eqref{m5} holds. Notably, there is some flexibility in choosing $s_\sharp$ above. However, the results obtained will be stronger if $s_\sharp$ is chosen to be `as large as possible' while still satisfying \eqref{m4}. In this work, we focus exclusively on the case $\alpha>0$ in \eqref{SP}, corresponding to a purely mixed local–nonlocal superposition operator. The case $\alpha=0$ will be addressed in forthcoming work.
\par In the present work, we focus on studying the regularity of weak solutions to the following problem
\begin{align}\label{MP} \tag{P}
    A_{\mu, p} u &= 0 \quad \text{in } \Omega,
\end{align}
where $p \in (1,\infty)$ and $\Omega$ is a bounded domain in $\mathbb{R}^N$.

  {\it A distinctive feature of these operators is that they simultaneously encode nonlinear effects and an infinite (possibly uncountable) family of fractional operators.} Some of the interesting examples for the measure $\mu$ over $(0,1)$ and $\alpha>0$, which are covered in this paper are as follows:
\begin{enumerate}
    \item When $\mu=\delta_s$ ( where $\delta_s$ represent the Dirac measure centred at $s\in(0,1)$) and $\alpha=1$, then mixed local-nonlocal operators and mixed local-nonlocal $p$-Laplace operators are $( -\Delta+(-\Delta)^s)$ and $( -\Delta_{p} +(-\Delta)_p^s)$.
    \item When $\mu=\epsilon\delta_s$ with $\epsilon\in (0,1)$, then mixed local-nonlocal operators is $ -\alpha \Delta_{p} +\epsilon(-\Delta_{p})^s$.
    \item When $\mu= a_1\delta_{s_1}+a_2\delta_{s_2}$ with $a_1,a_2 \in (0, +\infty)$ and $1>s_1>s_2> 0,$ then mixed local-nonlocal operators and mixed local-nonlocal operators $-\alpha \Delta_p+ a_1 (-\Delta)_p^{s_1}+ a_2 (-\Delta)_p^{s_2}$.
    \item When $\mu= \sum_{k=0}^{n} a_k \delta_{s_k}$ with $n\geq2$ such that  $a_k > 0$ and $1 >s_n>\ldots >s_1>s_0> 0,$ then mixed local-nonlocal operators associated with a convergent series of Dirac measures is $-\alpha\Delta_p+\sum_{k=0}^{n} a_k (-\Delta)_p^{s_k}$.
    \item When $\mu= \sum_{k=0}^{+\infty} a_k \delta_{s_k}$ with $\sum_{k=0}^{+\infty} a_k \in (0, +\infty)$ such that $a_0>0$ and $a_k \geq 0$ and $1 >\ldots s_2>s_1>s_0> 0,$ then mixed local-nonlocal operators associated with a convergent series of Dirac measures is $-\alpha\Delta_p+\sum_{k=0}^{+\infty} a_k (-\Delta)_p^{s_k}$.
    \item When $\d\mu(s)=f(s) \d s$ with $\bar{s}=\inf_{s\in \Sigma}s>0 $, where $f (\neq0)$ is a non-negative, measurable function, then mixed local-nonlocal operators (sum of uncountable operator) is $-\alpha \Delta_p+\int_{(0,1)} f(s)(-\Delta)_p^s \d s$, where $\d s$ represents the Lebesgue measure.
    \item When $\supp \{\mu\} \subset (0,1]$, then mixed local-nonlocal operators is given by $$\int_{(0,1)\cap \,\supp \{\mu\}}(-\Delta)_{p}^{s} u \d \mu(s)-\alpha \Delta_p u.$$
\end{enumerate}


Before stating the main results, we highlight the primary novelties of this work. The normalized superposition tail \eqref{tailinro} is introduced here for the first time and turns out to be particularly well suited for establishing local H\"older continuity. We believe that this notion will also play an important role in future investigations of the regularity theory for mixed-order superposition operators and related nonlocal problems. 
Another significant aspect of this work is that the results apply simultaneously to a large class of superposition operators involving an infinite (possibly uncountable) family of fractional phases. In particular, all the results obtained in this paper are completely new for the examples $(3)$--$(7)$ listed above, which exhibit genuinely mixed local-nonlocal and mixed-order behaviors. The lack of a single differentiability scale and the interaction between different fractional orders make these examples substantially more delicate than the classical fractional or local settings.

We now state the main results obtained in this paper. Our first result is the following Caccioppoli-type inequality for the weak solutions to \eqref{MP}, which is crucial in establishing the local boundedness of solutions. 
\begin{thm}[Caccioppoli-type inequality] \label{caccintro}  Let $u$ is a weak subsolution of \eqref{MP} and $\phi=(u-k)_+:=\max\{u-k, 0\}$ with $k\in \mathbb{R}$. Then there exists a positive constant $C:=C(p)$ such that
    \begin{align}\label{eq3.1intro}
     &\alpha \int_{B_r(x_0)} w^p|\nabla \phi|^p \d x+ \int_{(0,1)} C_{N,s,p} \left( \iint_{B_r(x_0)\times B_r(x_0)} \frac{|\phi(x) w(x)-\phi(y)w(y)|^p}{|x-y|^{N+sp}}\d x \d y\right) \d \mu \nonumber \\
   \quad\quad \leq &  C \Bigg[ \int_{(0,1)} C_{N,s,p} \left( \iint_{B_r(x_0)\times B_r(x_0)} \frac{ \max\{\phi(x), \phi(y)\}^p|w(x)-w(y)|^p}{|x-y|^{N+sp}}\d x\d y\right)\d \mu \nonumber \\
  & + \int_{(0,1)} C_{N,s,p} \left( \ess_{x\in \supp (w)} \int_{\mathbb{R}^N\setminus B_r(x_0)} \frac{\phi^{p-1}(y)}{|x-y|^{N+sp}}\d y \cdot \int_{B_r(x_0)}\phi(x) w^p(x) \d x \right)\d \mu \nonumber\\
  &\quad\quad\quad\quad+\alpha \int_{B_r(x_0)} \phi^p|\nabla w|^p \d x \Bigg], 
    \end{align}
    where $w$ is a nonnegative function such that $w\in C_c^\infty (B_r(x_0))$ and $B_r(x_0)\subset \Omega$.
\end{thm}
\begin{rem}
    Note that the estimate \eqref{eq3.1intro} continues to hold for $\phi=(u-k)_-$ whenever $u$ is a weak supersolution of \eqref{MP}.
\end{rem} 

The following theorem proves the local boundedness of weak subsolutions to \eqref{MP}. The proof substantially relies on the Sobolev inequality \eqref{sob}.

\begin{thm}[Local boundedness]  \label{locinyro}  Let $u$ be a weak subsolution to the problem \eqref{MP} and $p\in(1,\infty)$. There exists a constant $C:=C(N,p,\Sigma,\mu)>0$ such that
    \begin{equation}
        \ess_{B_{\frac{r}{2}(x_0)}} u \leq
            \delta \operatornamewithlimits{Tail}\bigg(u_+;x_0,\frac{r}{2}\bigg)+C\delta^{-\frac{(p-1)\eta}{(\eta-1)p}}\bigg( \fint_{B_r(x_0)}u_+^p \d x \bigg)^\frac{1}{p}, \nonumber
    \end{equation}
    where $B_r=B_r(x_0)\subset \Omega$ with $r\in (0,1]$, $\delta\in (0,1]$, $\eta$ defined in \eqref{grad sob} and $\Sigma:=\operatorname{supp}\{\mu\}$.
    
\end{thm}

In the following theorem, we prove the H\"older continuity of solutions, which extends the De Giorgi-Nash-Moser theory (see \cite{DeG57, Mos61, Nash58}) to the superposition operators of mixed fractional order. It is intrinsic to impose an extra condition $\bar{s}=\inf_{s\in \Sigma}s\in (0,1)$ to obtain the H\"older regularity, taking into account that $\mu(\{0\})=0$.

\begin{thm}[Local H\"older continuity]  Let $u$ be a weak solution of \eqref{MP}. Then $u$ is locally H\"older continuous in $\Omega$. Moreover, there exists $\sigma\in(0, \frac{\bar{s}p}{p-1})$ and positive constant $C:=C(N,p,\Sigma,\mu)$ such that
    \begin{equation*}
        \operatorname{osc}_{B_\epsilon(x_0)}u\leq C\left(\frac{\epsilon}{r} \right)^\sigma \left(  \operatorname{Tail}(u;x_0,{r})+C\left(\fint_{B_{2r}(x_0)}|u|^p \d x \right)^\frac{1}{p} \right), 
    \end{equation*}
    where $B_{2r}(x_0)\subset \Omega$ such that $r\in (0,1]$ and $\epsilon\in (0,r]$.
    
\end{thm}

The proof of H\"older continuity relies on the local boundedness estimates (Theorem \ref{locinyro}), the Caccioppoli-type inequality (Theorem \ref{caccintro}), and the logarithmic energy estimate established in \cite{AGKR2025i}. Our argument follows the approach developed in \cite{DKP2016}.

Our next result is the weak Harnack inequality for weak supersolutions.

\begin{thm}[Weak Harnack inequality] Let $u$ be a weak supersolution of \eqref{MP} with $u\geq 0$ in $B_R(x_0) \subset \Omega$. Then there exists a constant $C:=C(N,p,s,\Sigma)>0$ such that
    \begin{align*}
        \bigg(\fint_{B_{\frac{r}{2}}(x_0)}u^t \d x \bigg)^\frac{1}{t} \leq C\, \essi_{B_r(x_0)}u +C \,{\sup_{s\in \Sigma}\bigg(\frac{r}{R}\bigg)^{\frac{sp}{p-1}}} \operatorname{Tail}\big(u_-(y);x_0,R\big),
    \end{align*}
    where $r\in(0,1]$, $B_r(x_0)\subset B_{\frac{R}{2}}(x_0)$ and $0<t<\eta(p-1)$.
    \end{thm}

\begin{rem}
Although we establish a weak Harnack inequality in this work, we do not derive a full Harnack inequality for the problem \eqref{MP}. The main obstruction arises from the mixed local--nonlocal structure of the superposition operator, particularly from the tail contributions associated with the nonlocal component. The standard nonlocal tail given in \eqref{tailinro} is not sufficiently adapted to this superposition framework. To obtain a full Harnack inequality, it is necessary to introduce a slightly modified notion of tail capable of capturing the interaction between the local and nonlocal parts of the operator in a more refined manner. This requires a substantially different approach, which will be developed in a forthcoming work \cite{BGKL-harnack26}.
\end{rem}

Our final result concerns the lower and upper semicontinuity of weak supersolutions and subsolutions, respectively (see also Corollary~\ref{uppsem}). The proof is based on the approach developed in \cite{L2021}, adapted to the framework of superposition operators.
\begin{thm}[Lower semicontinuity]  Let $u$ be a weak supersolution of \eqref{MP}. Then
    \begin{align*}
        u(x)=u_*(x)  \text{ for every } x\in \mathcal{F}.
    \end{align*}
     In particular, $u_*$ is a lower semicontinuous representation of $u$ in $\Omega$.
\end{thm} 
The paper is organized as follows. In Section \ref{pre}, we recall several preliminary results, establish auxiliary lemmas, develop the tools needed for the appropriate solution space associated with problem \eqref{MP}, and introduce the notion of weak solutions to \eqref{MP}. In Section \ref{sec3}, we derive an energy estimate leading to a Caccioppoli-type inequality with tail and present some observations related to the logarithmic estimate for the problem under consideration. Section \ref{sec4} is devoted to proving the local boundedness of weak subsolutions to \eqref{MP}. In Section \ref{sec5}, we establish the H\"older continuity of weak solutions to \eqref{MP}. Section \ref{sec6} is concerned with the proof of a weak Harnack inequality for weak supersolutions of \eqref{MP}. Finally, in Section \ref{sec7}, we prove the lower and upper semicontinuity of weak supersolutions to \eqref{MP}.




\section{Function Space Setup and Preliminary Results} \label{pre}

The main objective of this section is to establish the functional analytic framework that focuses on the appropriate notions of fractional Sobolev spaces and their various properties, which are essential for studying our problem. For further details on this material, we refer to \cite{DPSV, DPSV2, DPSV1}. Additionally, we emphasize that some of the results discussed in this section are, to the best of our knowledge, new in the literature, and we provide their proofs.

We know the Gagliardo seminorm of $u$, is defined as
\begin{equation*}
    [u]_{s, p}:= \left(C_{N, s, p} \iint_{\mathbb{R}^{2 N}} \frac{|u(x)-u(y)|^{p}}{|x-y|^{N+s p}} \d x \d y\right)^{1 / p}  \text { for } s \in(0,1), 
\end{equation*}

where $$C_{N,p,s}:=\frac{\frac{sp}{2}(1-s)2^{2s-1}}{\pi^{\frac{N-1}{2}}}\frac{\Gamma(\frac{N+ps}{2})}{\Gamma(\frac{p+1}{2})\Gamma(2-s)} $$ is the normalizing constant. 
Thanks to the normalizing constant $C_{N, s, p},$ we have that
$$
\lim _{s \searrow 0}[u]_{s, p}=\|u\|_{L^{p}\left(\mathbb{R}^{N}\right)} \quad \text { and } \quad \lim _{s \nearrow 1}[u]_{s, p}=\|\nabla u\|_{L^{p}\left(\mathbb{R}^{N}\right)}.
$$


{Since, $\mu$ is a nonnegative, nontrivial finite Borel measure over $(0,1)$, $0<s<1$ and $1< p< \infty.$ We define the following function space when $\alpha= 0$:
 $$W^{s,p}_\mu(\mathbb{R}^N):=\{u:\mathbb{R}^N\rightarrow \mathbb{R}\text{ measurable }: \|u\|_\mu<+\infty \},$$
    equipped with the norm
    \begin{align}\label{eq2.2}
        \|u\|_\mu:=\bigg(\|u\|^p_{L^p(\mathbb{R}^N)}+\int_{(0,1)}[u]_{s,p}^p\d \mu(s)\bigg)^\frac{1}{p}.
    \end{align}
}
\begin{lem}
    The space $W^{s,p}_\mu(\mathbb{R}^N)$ is a Banach space with respect to the norm \eqref{eq2.2}.
\end{lem}
\begin{proof}
    Let $(u_k)$ be Cauchy sequence in $W^{s,p}_\mu(\mathbb{R}^N)$ with respect to the norm \eqref{eq2.2}. Then $(u_k)$ is a Cauchy sequence in $L^p(\mathbb{R}^N)$. Therefore, there exists $u\in L^p(\mathbb{R}^N)$ such that (up to subsequence)
    \begin{align*}
        u_k &\rightarrow u \text{ in } L^p(\mathbb{R}^N), \text{ and } 
        u_k(x) \rightarrow u(x)~ \text{a.e} \text{ in } \mathbb{R}^N.
    \end{align*}
    Since, $(u_k)$ is a Cauchy sequence in $W^{s,p}_\mu(\mathbb{R}^N)$, then for any $\epsilon>0$, there exists $M:=M(\epsilon)\in \mathbb{N}$ such that for all $k\geq M$,
    \begin{align}\label{eq2.260'}
        \epsilon^p \geq& \liminf_{h\rightarrow \infty} \|u_k-u_h\|_\mu \nonumber \\
        \geq &\liminf_{h\rightarrow \infty} \int_{\mathbb{R}^N} |u_k-u_h|^p \d x\nonumber \\
        &+ \liminf_{h\rightarrow \infty}\int_{(0,1)} C_{N,p,s} \bigg(\iint_{\mathbb{R}^{2N}} \frac{|(u_k-u_h)(x)-(u_k-u_h)(y)|^p}{|x-y|^{{N+sp}}} \d x \d y \bigg) \d \mu.
    \end{align}
    Again, $u_k(x)\rightarrow u(x)$ a.e in $\mathbb{R}^N$. Thus, using Fatou's lemma in \eqref{eq2.260'}, we obtain 
    \begin{align*}\label{eq2.270'}
        \epsilon^p \geq& \int_{\mathbb{R}^N} |u_k-u|^p \d x+ \int_{(0,1)} C_{N,p,s} \bigg(\iint_{\mathbb{R}^{2N}} \frac{|(u_k-u)(x)-(u_k-u)(y)|^p}{|x-y|^{{N+sp}}} \d x \d y \bigg) \d \mu \nonumber \\
        =&\|u_k-u\|_\mu~ \text{  for all } k\geq M.
    \end{align*}
    Hence, $u_k\rightarrow u$ in $W^{s,p}_{\mu}(\mathbb{R}^N)$. This completes the proof.
\end{proof}
\begin{definition}\label{Tail}
    Let {$u\in L_{loc}^{p-1}(\mathbb{R}^N)$} and $B_r(x_0)\subset \mathbb{R}^N$ be an open ball. The nonlocal tail of $u$ concerning the ball $B_r(x_0)$ is defined as
    \begin{equation*}
        \operatorname{Tail}(u;x_0,r)=\left[\int_{(0,1)}{C_{N,s,p}}r^{sp}\left(\int_{\mathbb{R}^N\setminus B_r(x_0)}\frac{|u(y)|^{p-1}}{|y-x_0|^{N+sp}}dy\right)d \mu(s)\right]^{\frac{1}{p-1}}.
    \end{equation*}
\end{definition}
\begin{rem}
    It is noteworthy to mention that we can define ``tail" without the normalizing constant $C_{N,p,s}$. The role of the normalizing constant $C_{N,p,s}$ is precisely used to obtain the H\"older continuity of solutions (refer to the proof in Section \ref{sec5}). 
\end{rem}
We now define the `tail space' as follows
\begin{align*}
    L^{p-1}_{sp,\mu}(\mathbb{R}^N):=\{u\in L_{loc}^{p-1}(\mathbb{R}^N):\operatorname{Tail}(u;x,r)<\infty, \forall~ x\in \mathbb{R}^N, \forall ~r\in(0,\infty)\}.
\end{align*}
Observe that {$W^{s,p}_{\mu}(\mathbb{R}^N) \subset L^{p-1}_{sp,\mu}(\mathbb{R}^N).$} 

The next lemma is derived by adapting ideas given in \cite[Proposition $2.1$]{NPV2012}.
\begin{lem} 
Let $\Omega$ be a bounded domain in $\mathbb{R}^N$, $s\in(0,1)$ and $p\in(1,\infty)$, then there exists $C(N,\Omega,p)>0$, independent of $s$, such that
\begin{align*}
    \int_{(0,1)}[u]_{s,p}^p\d \mu(s) \leq C \int_\Omega |\nabla u(x)|^p \d x, \text{ for every } u\in W_0^{1,p}(\Omega).
\end{align*}
\end{lem}
\begin{proof}
    We first rewrite the following integral
    \begin{align*}
       &\int_{(0,1)} C_{N, s, p}\left( \iint_{\mathbb{R}^{2 N}} \frac{|u(x)-u(y)|^{p}}{|x-y|^{N+s p}} \d x \d y\right) \d \mu\nonumber \\
       =&2\int_{(0,1)} C_{N, s, p}\left( \int_\Omega \int_{\mathbb{R}^N\setminus \Omega} \frac{|u(x)-u(y)|^{p}}{|x-y|^{N+s p}} \d x \d y\right) \d \mu \nonumber \\
       &+\int_{(0,1)} C_{N, s, p} \left(\int_\Omega  \int_\Omega \frac{|u(x)-u(y)|^{p}}{|x-y|^{N+s p}} \d x \d y\right) \d \mu :=2J_1+J_2.
    \end{align*}
   We provide the detailed computation for $J_1$. By a similar argument, the conclusion holds for $J_2$ as well. First, using the convexity property together with the Poincar\'e inequality, we obtain
    \begin{align}\label{eq2.38}
        &\int_{(0,1)} C_{N, s, p}\left( \int_\Omega \int_{\{\mathbb{R}^N\setminus \Omega\} \cap \{|x-y| \geq 1\}} \frac{|u(x)-u(y)|^{p}}{|x-y|^{N+s p}} \d x \d y\right) \d \mu \nonumber \\
        & \leq C(p) \int_{(0,1)}C_{N, s, p} \left( \int_\Omega \int_{ \{|x-y| \geq 1\}} \frac{|u(x)|^p+|u(y)|^{p}}{|x-y|^{N+s p}} \d x \d y\right) \d \mu \nonumber \\
        & \leq C(p) \int_{(0,1)}C_{N, s, p} \bigg( \int_\Omega \bigg(\int_{ |z| \geq 1} \frac{1}{|z|^{N+sp}} \d z \bigg)|u(x)|^p \d x \bigg)\d \mu \nonumber \\
        & \leq C(p,N)  \int_{(0,1)} C_{N, s, p} \bigg( \frac{1}{s} \|u\|^p_{L^p(\Omega)} \bigg)\d \mu \nonumber \\
         & \leq C(p,N,\Omega)  \int_{(0,1)} C_{N, s, p} \bigg( \frac{1}{s}  \|\nabla u\|^p_{L^p(\Omega)}\bigg)\d \mu.
    \end{align}
    On the other hand, by using H\"older inequality, we obtain 
    \begin{align}\label{eq2.39}
         &\int_{(0,1)}C_{N, s, p} \left( \int_\Omega \int_{\{\mathbb{R}^N\setminus \Omega\} \cap \{|x-y| < 1\}} \frac{|u(x)-u(y)|^{p}}{|x-y|^{N+s p}} \d x \d y\right) \d \mu \nonumber \\
         \leq & \int_{(0,1)}C_{N, s, p} \bigg( \int_\Omega \int_{|z|<r} \frac{|u(x)-u(z+x)|^p}{|z|^{N+s p}} \d z \d x \bigg) \d \mu \nonumber \\
          \leq & \int_{(0,1)}C_{N, s, p} \bigg( \int_\Omega \int_{|z|<1} \bigg(\int_0^1|\nabla(x+tz)|\d t \bigg)^p \frac{1}{|z|^{N+sp-p}}\d z \d x \bigg) \d \mu \nonumber \\
         \leq & \int_{(0,1)}C_{N, s, p} \bigg( \int_{\mathbb{R}^N} \int_{|z|<1} \int_0^1 \frac{|\nabla(x+tz)|^p}{|z|^{N+sp-p}} \d t \d z \d x \bigg) \d \mu \nonumber \\
         \leq & \int_{(0,1)} C_{N, s, p}\bigg( \int_{|z|<1} \int_0^1 \frac{\|\nabla u\|_{L^p(\mathbb{R}^N)}^p}{|z|^{N+sp-p}} \d t \d z  \bigg) \d \mu \nonumber \\
         \leq &C(p,N) \int_{(0,1)} C_{N, s, p}\bigg( \frac{1}{1-s}\|\nabla u\|_{L^p(\Omega)}^p \bigg) \d \mu.
    \end{align}
    Now from \eqref{eq2.38} and \eqref{eq2.39}, we derive
    \begin{align*}
        J_1\leq &C(p,N,\Omega) \int_{(0,1)} C_{N, s, p}\bigg( \frac{1}{1-s}+\frac{1}{s}\bigg)\|\nabla u\|_{L^p(\Omega)}^p  \d \mu \nonumber \\
        &\leq C(p,N,\Omega)\|\nabla u\|_{L^p(\Omega)}^p.
    \end{align*}
    This completes the proof.
\end{proof}

We now recall the following Sobolev inequality (see \cite[Corollary 1.57]{MZ97}), which will play a crucial role in our analysis.
\begin{lem}
    Let $\Omega$ be a open bounded subset of $\mathbb{R}^N$, $p\in(1,\infty)$ and 
    \begin{align}\label{grad sob}
        \eta=\begin{cases}
            & \frac{N}{N-p}~ \text{when } p\in(1,N),\\
            & 2 \text{ when } p\in [N,\infty).
        \end{cases}
    \end{align}
    Then there exists a constant $C=C(p,N)>0$ such that 
    \begin{align}\label{sob}
       \left( \int_\Omega |u(x)|^{\eta p}\d x\right)^\frac{1}{\eta p}\leq C|\Omega|^{\frac{1}{N}-\frac{1}{p}+\frac{1}{\eta p}}\left(\int_\Omega|\nabla u(x)|^p \d x\right)^\frac{1}{p}, \text{ for any } u\in W_{0}^{1,p}(\Omega).
    \end{align}
   
\end{lem}
We now introduce the notion of a weak solution for problem \eqref{MP}.
\begin{definition}[Weak solution]\label{def2.11}
    A function $u\in W_{loc}^{1,p}(\Omega)\cap L^{p-1}_{sp,\mu}(\mathbb{R}^N)$ is called a weak subsolution of \eqref{MP} if, for every $\Omega'\Subset \Omega$ and every nonnegative function $v\in W_0^{1,p}(\Omega')$ such that
\begin{align}\label{eq2.10} 
\int_{(0,1)}C_{N, s, p} &\left(\iint_{\mathbb{R}^{2 N}} \frac{|u(x)-u(y)|^{p-2}(u(x)-u(y))(v(x)-v(y))}{|x-y|^{N+s p}} \d x \d y\right) \d \mu(s)\nonumber \\ 
&+ \alpha \int_{\Omega} |\nabla u(x)|^{p-2}\nabla u(x) \cdot \nabla v(x) \d x \leq 0. 
\end{align} 
Similarly, a function $u \in W_{loc}^{1,p}(\Omega)\cap L^{p-1}_{sp,\mu}(\mathbb{R}^N)$ is called a weak supersolution of \eqref{MP} if the integral in \eqref{eq2.10} is nonnegative for every $\Omega'\Subset \Omega$ and every nonnegative function $v\in W^{1,p}_0(\Omega')$.

Moreover, a function $u\in W_{loc}^{1,p}(\Omega)\cap L^{p-1}_{sp,\mu}(\mathbb{R}^N)$ is called a weak solution of \eqref{MP} if the integral in \eqref{eq2.10} is equal to zero for every $v\in W^{1,p}_0(\Omega')$.
\end{definition}
    
It is important to observe that, in view of Definition \ref{def2.11}, the following properties hold:
\begin{itemize}
\item[(i)] $u$ is a weak subsolution of \eqref{MP} if and only if $-u$ is a weak supersolution of \eqref{MP}.

\item[(ii)] For any $c\in \mathbb{R}$, the function $u+c$ is a weak solution of \eqref{MP} if and only if $u$ is a weak solution of \eqref{MP}$.$

\item[(iii)] $u$ is a weak solution of \eqref{MP} if and only if $-u$ is also a weak solution of \eqref{MP}$.$
\end{itemize}
Before proceeding to the proofs of our results, we introduce the following notations to be used throughout the paper for simplified presentation.\\

\noindent{\bf Notation:} We use the following notations:
   \begin{itemize}
       \item We denote $u_+=\max\{u,0\}$ and $u_-=\max\{-u,0\}=-\min\{u,0\}$ for any $u\in \mathbb{R}$, so that $u=u_+-u_-$ and $|u|=u_++u_-$.
       \item The barred integral sign $\fint$ denotes the corresponding integral average.
       \item $B_r(x_0)$ is the open ball centered at $x_0$ with radius $r>0$.
       \item We will use the following three notations for simplification of the fractional integral part of the operator whenever necessary: 
       \begin{align*}
           \mathcal{A}(u(x,y)):=&|u(x)-u(y)|^{p-2}(u(x)-u(y)) \text{  and  } \d \nu:=\frac{\d x \d y}{|x-y|^{N+sp}}.
       \end{align*}
       \item For a bounded domain $\Omega \subset \mathbb{R}^N$, $W_0^{1,p}(\Omega)$ denotes the standard Sobolev space, which is defined as 
\begin{align*}
    W_0^{1,p}(\Omega):=\{u\in W^{1,p}(\Omega):u=0 \text{ in } \mathbb{R}^N\setminus \Omega\}.
\end{align*}
   \end{itemize}

   We begin with the following lemma.
\begin{lem}\label{lmn2.15}
    A function $u$ is a weak subsolution and a weak supersolution of \eqref{MP} if and only if $u$ is a weak solution of \eqref{MP}.
\end{lem}
\begin{proof}
     Suppose that $u$ is a weak subsolution and a weak supersolution of \eqref{MP}. For every $\Omega'\Subset \Omega$ and  $v \in W_0^{1,p}(\Omega')$, then $v_+,v_-\in W_0^{1,p}(\Omega')$. Given that $u$ is a weak subsolution and choosing $v_+$ as the test function in the weak formulation \eqref{eq2.10}, we obtain
\begin{align}\label{eq2.11}
    \int_{(0,1)}C_{N, s, p}\left(\iint_{\mathbb{R}^{2 N}} {\mathcal{A}(u(x,y))(v_+(x)-v_+(y))} \d\nu\right) \d \mu(s) \nonumber\\
    +\alpha \int_{\Omega} |\nabla u(x)|^{p-2}\nabla u(x) \cdot \nabla v_+(x) \d x  \leq 0.
\end{align}
   Similarly, utilizing $u$ is a weak supersolution and choosing $v_-$ as the test function in the weak formulation \eqref{eq2.10}, we get
   \begin{align}\label{eq2.12}
    \int_{(0,1)}C_{N, s, p}\left(\iint_{\mathbb{R}^{2 N}} {\mathcal{A}(u(x,y))(v_-(x)-v_-(y))}\d \nu \right) \d \mu(s) \nonumber\\
    +\alpha \int_{\Omega} |\nabla u(x)|^{p-2}\nabla u(x) \cdot \nabla v_-(x) \d x \geq 0.
\end{align}
Using \eqref{eq2.11}, \eqref{eq2.12} and $v=v_+-v_-$, we deduce
\begin{align}\label{eq2.13}
    \int_{(0,1)}C_{N, s, p}\left(\iint_{\mathbb{R}^{2 N}} {\mathcal{A}(u(x,y))(v(x)-v(y))}\d \nu\right) \d \mu(s) \nonumber\\
   +\alpha \int_{\Omega} |\nabla u(x)|^{p-2}\nabla u(x) \cdot \nabla v(x) \d x \leq 0.
\end{align}
The opposite inequality holds by replacing $v$ with $-v$ in \eqref{eq2.13}. Therefore, $u$ is a weak solution of \eqref{MP}. Conversely, if $u$ is a weak solution of \eqref{MP}, then by Definition \ref{def2.11}, we conclude that $u$ is both a weak subsolution and a weak supersolution.
\end{proof}
\begin{lem}\label{lmn2.16}
If $u$ is a weak subsolution of \eqref{MP}, then $u_+$ is a weak subsolution of \eqref{MP}.
\end{lem}
\begin{proof}
    Let  $u_m=\min\{mu_+,1\}$, $m\in \mathbb{N}$ in $\Omega$. Then we have
    \begin{enumerate}[(i)]
        \item $(u_m)$ is an increasing sequence of functions,
        \item $0\leq u_m\leq 1$ for every $m\in \mathbb{N}$, and 
        \item $(u_m)$ converges to $1$ when $u_+>0$.
    \end{enumerate}
  For every nonnegative $v\in C_c^\infty(\Omega')$ we choose $u_mv\in W_0^{1,p}(\Omega')$ as a test function in weak formulation \eqref{eq2.10} to get
    \begin{align}\label{eq2.14}
        \int_{(0,1)}C_{N, s, p}\left(\iint_{\mathbb{R}^{2 N}} {\mathcal{A}(u(x,y))(u_m(x)v(x)-u_m(y)v(y))}\d \nu\right) \d \mu(s)\nonumber\\
        +\alpha \int_{\Omega} |\nabla u(x)|^{p-2}\nabla u(x) \cdot \nabla u_m(x)v(x) \d x \leq 0.
    \end{align}
    Let
    \begin{align*}
    I_1&=  \int_{(0,1)}C_{N, s, p} \bigg(\iint_{\mathbb{R}^{2 N}} {\mathcal{A}(u(x,y))(u_m(x)v(x)-u_m(y)v(y))}\d \nu\bigg) \d \mu(s), \text{ and }\\
       I_2&=\int_{\Omega} |\nabla u|^{p-2}\nabla u.\nabla (u_mv)\d x.
    \end{align*}
    We first estimate $I_1$ for all $x,y\in \mathbb{R}^N$. We divide the estimation into three cases.
    
    \noindent \textbf{ Case-I: When $u(x)>u(y)$.} 
    Observe that $u_m(x)=0$ implies $u_m(y)=0$. Thus we get
    \begin{align}\label{eq2.20}
        |u(x)-u(y)|^{p-2}(u(x)-u(y))(u_m(x)v(x)-u_m(y)v(y))=0.
    \end{align}
If $u_m(y)>0$, then we have $u(y)=u_+(y)$. Given the condition $u(x)>u(y)$ implies that $u(x)=u_+(x)$. Thus, $u_m(x)>u_m(y)$. Therefore,
     \begin{align}\label{eq2.21}
        &|u(x)-u(y)|^{p-2}(u(x)-u(y))(u_m(x)v(x)-u_m(y)v(y))\nonumber \\
        \geq & (u_+(x)-u_+(y))^{p-1}u_m(x)(v(x)-v(y)).
    \end{align}
 Again, $u_m(y)=0$ and $u_m(x)>0$ imply $u(y)\leq 0<u(x)$. Thus, we have
     \begin{align}\label{eq2.22}
        &|u(x)-u(y)|^{p-2}(u(x)-u(y))(u_m(x)v(x)-u_m(y)v(y))=(u(x)-u(y))^{p-1}u_m(x)v(x) \nonumber\\
        \geq &(u_+(x)-u_+(y))^{p-1}u_m(x)v(x) \nonumber \\
         \geq &(u_+(x)-u_+(y))^{p-1}u_m(x)(v(x)-v(y)).
    \end{align}
      Therefore, using \eqref{eq2.20}, \eqref{eq2.21} and \eqref{eq2.22}, we conclude that
      \begin{align}\label{eq2.23}
          &|u(x)-u(y)|^{p-2}(u(x)-u(y))(u_m(x)v(x)-u_m(y)v(y)) \nonumber \\
          \geq & (u_+(x)-u_+(y))^{p-1}u_m(x)(v(x)-v(y))
      \end{align}
\noindent \textbf{Case-II: When $u(x)=u(y)$.} In this case, we have $u_+(x)=u_+(y)$. Therefore, the estimate \eqref{eq2.23} holds.

\noindent \textbf{Case-III: When $u(x)<u(y)$.}       Interchanging the role of $x$ and $y$ in \textbf{Case-I}, we obtain
      \begin{align}\label{eq2.24}
          &|u(x)-u(y)|^{p-2}(u(x)-u(y))(u_m(x)v(x)-u_m(y)v(y)) \nonumber \\
          \geq &(u_+(y)-u_+(x))^{p-1}u_m(y)(v(y)-v(x))\nonumber \\
          =&(u_+(x)-u_+(y))^{p-1}u_m(x)(v(x)-v(y)).
      \end{align}

      We now estimate $I_2$ as below.
    \begin{align}\label{eq2.19}
        I_2=\int_{\Omega} |\nabla u|^{p-2}\nabla u.\nabla (u_mv)\d x &= \int_{\Omega}u_m |\nabla u|^{p-2}\nabla u.\nabla v\d x+m\int_{\Omega\cap \{0<u_+<\frac{1}{m}\}} |\nabla u|^p v \d x \nonumber\\
        &\geq \int_{\Omega}u_m |\nabla u|^{p-2}\nabla u.\nabla v\d x.
    \end{align}
      Therefore, using \eqref{eq2.23}, \eqref{eq2.24} and \eqref{eq2.19} in \eqref{eq2.14}, we derive 

      \begin{align}\label{eq2.25}
        \int_{(0,1)}C_{N, s, p} \left(\iint_{\mathbb{R}^{2 N}} {|u_+(x)-u_+(y)|^{p-2}(u_+(x)-u_+(y))(v(x)-v(y))}\d \nu \right) \d \mu(s) \nonumber\\
       +\alpha \int_{\Omega} |\nabla u_+|^{p-2}\nabla u_+.\nabla (v)\d x\leq 0.
    \end{align}
    Therefore, by using the density argument in \eqref{eq2.25} for every $v\in W_0^{1,p}(\Omega)$, we get the desired result that $u_+$ is a weak subsolution of \eqref{MP}. This completes the proof.
\end{proof}
{
\begin{lem}\label{lmn2.17}
    If $u$ is a weak supersolution of \eqref{MP}, then $u_-$ is a weak subsolution of \eqref{MP}. 
\end{lem}
\begin{proof}
    Since $u$ is a weak supersolution of \eqref{MP}, then $(-u)$ is a weak subsolution of \eqref{MP}. From Lemma \ref{lmn2.16}, we conclude that $u_-$ is a weak subsolution of \eqref{MP}.
\end{proof}
}
We now recall the following two lemmas, which are useful to prove boundedness, H\"older continuity, and the weak Harnack inequality.
\begin{lem}[Lemma $3.1$ \cite{DKP2016}]\label{lmn2.9}
   Let $p\geq 1$, $a,b\in \mathbb{R}^N$, and $\epsilon \in (0,1]$. Then
\begin{align*}
|a|^p\leq |b|^p+C\epsilon|b|^p+(1+C\epsilon)\epsilon^{1-p}|a-b|^p,
\end{align*}
where $C:=(p-1)\Gamma(\max{1,p-2})$ and $\Gamma$ denotes the Gamma function.
\end{lem}
\begin{lem}\cite[Lemma 4.1]{D2012}\label{MI Lemma}
    Let $\{P_j\}_{j=0}^\infty$ be a sequence of positive real numbers such that $P_{j+1}\leq C_1 C_2^j P^{1+\beta}_j$, where $P_0\leq C_1^{-\frac{1}{\beta}}C_2^{-\frac{1}{\beta^2}}$ and $C_1,C_2>1$, $\beta>0$ are some constants. Then $\lim_{j\rightarrow \infty} P_j=0.$
\end{lem}


\section{Energy Estimates} \label{sec3}
In this section, we establish several important energy estimates that will be used in the subsequent section. Our approach follows the method developed by Di Castro \textit{et al.} \cite[Theorem $1.4$]{DKP2016}.
\begin{thm}\label{lmn3.1}
    Let $u$ be a weak subsolution of \eqref{MP} and let $\phi=(u-k)_+$ for some $k\in \mathbb{R}$. Then there exists a positive constant $C:=C(p)$ such that
    \begin{align}\label{eq3.1}
     &\alpha \int_{B_r(x_0)} w^p|\nabla \phi|^p \d x+ \int_{(0,1)} C_{N,s,p} \left( \iint_{B_r(x_0)\times B_r(x_0)} \frac{|\phi(x) w(x)-\phi(y)w(y)|^p}{|x-y|^{N+sp}}\d x \d y\right) \d \mu \nonumber\\
    \leq &  C \Bigg[ \int_{(0,1)} C_{N,s,p} \left( \iint_{B_r(x_0)\times B_r(x_0)} \frac{ \max\{\phi(x), \phi(y)\}^p|w(x)-w(y)|^p}{|x-y|^{N+sp}}\d x\d y\right)\d \mu \nonumber \\
  & + \int_{(0,1)} C_{N,s,p} \left( \ess_{x\in \supp w} \int_{\mathbb{R}^N\setminus B_r(x_0)} \frac{\phi^{p-1}(y)}{|x-y|^{N+sp}}\d y \cdot \int_{B_r(x_0)}\phi(x) w^p(x) \d x \right)\d \mu \nonumber\\
  &+\alpha \int_{B_r(x_0)} \phi^p|\nabla w|^p \d x \Bigg],
    \end{align}
    where $w$ is a nonnegative function with $w\in C_c^\infty (B_r(x_0))$ and $B_r(x_0)\subset \Omega$.
\end{thm}
\begin{proof}
    Since $u$ is a weak subsolution of \eqref{MP} and $\phi=(u-k)_+$, by choosing $v=\phi w^p$ as a test function in \eqref{eq2.10}, we obtain
    \begin{align}\label{eq3.2}
        0&\geq  \alpha \int_{\Omega} |\nabla u(x)|^{p-2}\nabla u(x) \cdot \nabla (\phi(x) w^p(x)) \d x\nonumber \\
        +&\int_{(0,1)}C_{N, s, p}\left(\iint_{\mathbb{R}^{2 N}} {\mathcal{A}(u(x,y))(\phi(x) w^p(x)-\phi(y) w^p(y))} \d \nu \right) \d \mu(s):=\alpha I_1+I_2.
 \end{align}
 \noindent \textbf{Estimation of $I_1$:} By arguments analogous to those in \cite[Proposition 3.1]{BDL2021}, we obtain
\begin{align}\label{eq3.8}
I_1=\int_{\Omega} |\nabla u|^{p-2}\nabla u\cdot\nabla (\phi w^p)\d x
\geq C_1 \int_{B_r(x_0)} w^p|\nabla \phi|^{p}\d x
-C_2 \int_{B_r(x_0)} \phi^p|\nabla w|^{p}\d x,
\end{align}
for some positive constants $C_1$ and $C_2$ depending only on $p$.

\noindent \textbf{Estimation on $I_2$:} Using the property of $w$, we obtain
    \begin{align}\label{eq3.7}
        I_2=& \int_{(0,1)}C_{N, s, p}\bigg(\iint_{\mathbb{R}^{2 N}} {\mathcal{A}(u(x,y))(\phi(x) w^p(x)-\phi(y) w^p(y))}\d \nu \bigg) \d \mu(s)\nonumber \\
=&\int_{(0,1)} C_{N, s, p}\left(\iint_{B_r(x_0)\times B_r(x_0)} {\mathcal{A}(u(x,y))(\phi(x)w^p(x)-\phi(y)w^p(y))}\d \nu\right) \d \mu(s)\nonumber\\
+& 2\int_{(0,1)}C_{N, s, p} \left(\int_{\mathbb{R}^{N}\setminus B_r(x_0)}\int_{B_r(x_0)} {\mathcal{A}(u(x,y))\phi(x)w^p(x)}\d \nu \right) \d \mu(s):=I_2'+2I_2''
\end{align}
\noindent  \textbf{Estimation on $I_2'$:} Without loss of generality, we assume that $u(x)\geq u(y)$; otherwise, we interchange the roles of $x$ and $y$. Observe that
\begin{align}\label{eq350}
    &\mathcal{A}(u(x,y))\{\phi(x)w^p(x)-\phi(y)w^p(y)\}\nonumber\\
    =&(u(x)-u(y))^{p-1}\{(u(x)-k)_+w^p(x)-(u(y)-k)_+w^p(y)\} \nonumber \\
    \geq &(\phi(x)-\phi(y))^{p-1} \{\phi(x)w^p(x)-\phi(y)w^p(y)\}.
\end{align}
Now if $\phi(x)\geq \phi(y)$ and $w(y)\geq w(x)$, using Lemma \ref{lmn2.9}, we get
\begin{align}\label{eq360}
    w^p(x)\geq (1-C\epsilon)w^p(y)-(1+C\epsilon)\epsilon^{1-p}|w(x)-w(y)|^p.
\end{align}
Choosing \begin{align}
    \epsilon:=\frac{\phi(x)-\phi(y)}{\max\{1,2C\}\phi(x)} \in(0,1];
\end{align}
 and using \eqref{eq360}, we obtain 
\begin{align}\label{eq380}
   & (\phi(x)-\phi(y))^{p-1}\phi(x)w^p(x)\geq (\phi(x)-\phi(y))^{p-1}\phi(x)\max\{w(x),w(y)\}^p \nonumber \\
    &-\frac{1}{2}(\phi(x)-\phi(y))^{p}\max\{w(x),w(y)\}^p-C(p)\max\{\phi(x),\phi(y)\}^p|w(x)-w(y)|^p.
\end{align}
Observe that when $\phi(x)\geq \phi(y)$ and $w(y)\leq w(x)$, we get the same estimate \eqref{eq380} trivially. Thus for $\phi(x)\geq \phi(y)$, we get
\begin{align}\label{eq390}
   & (\phi(x)-\phi(y))^{p-1} \{\phi(x)w^p(x)-\phi(y)w^p(y)\}\nonumber \\
    \geq &(\phi(x)-\phi(y))^{p-1}\{\phi(x)\max\{w(x),w(y)\}^p -\phi(y)w^p(y)\}\nonumber \\
    &-\frac{1}{2}(\phi(x)-\phi(y))^{p}\max\{w(x),w(y)\}^p-C\max\{\phi(x),\phi(y)\}^p|w(x)-w(y)|^p \nonumber \\
    \geq& \frac{1}{2}(\phi(x)-\phi(y))^{p}\max\{w(x),w(y)\}^p-C\max\{\phi(x),\phi(y)\}^p|w(x)-w(y)|^p.
\end{align}
For the case $\phi(x)< \phi(y)$, we may interchange the roles of $x$ and $y$ to obtain the same estimate as in \eqref{eq390}.
Furthermore, by the convexity property, we obtain
\begin{align}\label{eq310}
    |\phi(x) w(x)-\phi(y)w(y)|^p \leq& 2^{p-1}|\phi(x)-\phi(y)|^p \max\{w(x),w(y) \}^p\nonumber \\
   & +2^{p-1}\max\{\phi(x),\phi(y)\}^p|w(x)-w(y)|^p.
\end{align}
Finally using \eqref{eq350}, \eqref{eq390} and \eqref{eq310}, we obtain
\begin{align}\label{eq3111}
  I_2'&\geq C_3 \int_{(0,1)}C_{N,s,p}\bigg(  \iint_{B_r(x_0)\times B_r(x_0)} {|\phi(x) w(x)-\phi(y)w(y)|^p}\d \nu \bigg) \d \mu(s)\nonumber\\
       & - C_4 \int_{(0,1)}C_{N,s,p}\bigg(  \iint_{B_r(x_0)\times B_r(x_0)} { \max\{\phi(x), \phi(y)\}^p|w(x)-w(y)|^p}\d \nu \bigg) \d \mu(s),
\end{align}
 for some positive constants $C_3 \text{ and } C_4$, depending only on $p$.
 
\noindent \textbf{Estimation on $I_2''$:} Observe that $$\mathcal{A}(u(x,y))\phi(x)=|u(x)-u(y)|^{p-2}(u(x)-u(y))(u(x)-k)_+\geq -\phi^{p-1}(y)\phi(x).$$ 
Using this fact, we obtain
\begin{align}\label{eq3112}
    I_2''\geq& -\int_{(0,1)}C_{N, s, p} \left(\int_{\mathbb{R}^{N}\setminus B_r(x_0)}\int_{B_r(x_0)} \frac{\phi^{p-1}(y)\phi(x) w^p(x)}{|x-y|^{N+s p}} \d x \d y\right) \d \mu(s) \nonumber \\
    \geq& -{ C_5 \int_{(0,1)}C_{N,s,p}\bigg( \mathop{\mathrm{\ess}}_{x \in \supp w} \int_{\mathbb{R}^N\setminus B_r(x_0)} \frac{\phi^{p-1}(y)}{|x-y|^{N+sp}}\d y \cdot \int_{B_r(x_0)}\phi w^p \d x \bigg) \d \mu(s),}
\end{align}
for some positive constant $C_5$, depending only on $p$.
   Now, substituting \eqref{eq3.8}, \eqref{eq3.7}, \eqref{eq3111}, and \eqref{eq3112} into \eqref{eq3.2}, we obtain the result \eqref{eq3.1}.
\end{proof}
Using Lemma \ref{lmn3.1}, we derive the following energy estimate.
\begin{cor}\label{cor3.2}
    
    Let $u$ be a weak subsolution of \eqref{MP} and let $\phi=(u-k)_+$ for some $k\in \mathbb{R}$. Then there exists a positive constant $C:=C(p)$ such that
    \begin{align}
     &\alpha \int_{B_r(x_0)} w^p|\nabla \phi|^p \d x \leq C \Bigg[ \alpha \int_{B_r(x_0)} \phi^p|\nabla w|^p \d x \nonumber \\
     & +  \int_{(0,1)} C_{N,s,p} \left( \iint_{B_r(x_0)\times B_r(x_0)} \frac{ \max\{\phi(x), \phi(y)\}^p|w(x)-w(y)|^p}{|x-y|^{N+sp}}\d x\d y\right)\d \mu \nonumber \\
  & + \int_{(0,1)} C_{N,s,p} \left( \ess_{x\in \supp w} \int_{\mathbb{R}^N\setminus B_r(x_0)} \frac{\phi^{p-1}(y)}{|x-y|^{N+sp}}\d y \cdot \int_{B_r(x_0)}\phi(x) w^p(x) \d x \right)\d \mu \Bigg], \nonumber
    \end{align}
    where $w$ is a nonnegative function with $w\in C_c^\infty (B_r(x_0))$ and $B_r(x_0)\subset \Omega$.
\end{cor}
Using arguments analogous to those in Lemma \ref{lmn3.1}, we obtain the following lemmas:
\begin{lem}\label{lmn3.2}
Let $u$ be a weak supersolution of \eqref{MP} and let $\phi=(u-k)_-$ with $k\in \mathbb{R}$. Then the estimate in \eqref{eq3.1} holds.
\end{lem}

\begin{lem}\label{lmn3.3}
Let $u$ be a weak solution of \eqref{MP} and let $\phi=(u-k)_\pm$ with $k\in \mathbb{R}$. Then the estimate in \eqref{eq3.1} holds.
\end{lem}
Using an approach analogous to that of Lemma $5.1$ in \cite{AGKR2025i}, one can establish the following logarithmic lemma:
\begin{lem}\label{lmn3.4}
  Let $p\in (1,\infty)$ and $\Sigma:=\operatorname{supp}\{\mu\}$. Let $u$ be a weak {supersolution} of \eqref{MP} such that $u\geq 0$ in $B_R(x_0)\subset \Omega$. Then for any $B_r:=B_r(x_0)\subset B_{\frac{R}{2}}(x_0)$ and $d>0$, there exists a constant $C=C(N,p,\Sigma,\mu)>0$ such that
  \begin{align*}
     & \alpha \int_{B_r}|\nabla \log(u+d)|^p \d x+\int_{(0,1)}C_{N,p,s} \bigg(\int_{B_r}\int_{B_r}\left|\log\left(\frac{u(x)+d}{u(y)+d}\right)\right|^{p}\d \nu \bigg) \d \mu(s)\nonumber\\
        &\leq  C r^N \sup\limits_{s\in\Sigma} r^{-sp}\bigg(d^{1-p} {\sup\limits_{s\in\Sigma}\left(\frac{r}{R}\right)^{sp}}[\operatorname{Tail}(u_-;x_0,R)]^{p-1}+1\bigg)+C\alpha r^{N-p}.
  \end{align*}
\end{lem}
\begin{cor}\label{coro}
    Let $u$ be a weak solution of \eqref{MP} such that $u\geq 0$ in $B_R:=B_R(x_0) \subset \Omega$ and $b,d>0,~ a>1$. Let us define
    \begin{align}\nonumber
        \phi= \min\left\{\left(\log\left(\frac{b+d}{u+d}\right)\right)_+, \ \log a\right\}.
    \end{align}
    Then there exists a constant $C=C(N,p, \Sigma, \mu)>0$ such that 
    \begin{align*}
        \fint_{B_r(x_0)}\left|\phi-(\phi)_{B_r(x_0)}\right|^p \d x \leq C  \bigg(d^{1-p} \sup\limits_{s\in\Sigma}\left(\frac{r}{R}\right)^{sp}[\operatorname{Tail}(u_-;x_0,R)]^{p-1}+1\bigg),
    \end{align*}
    where $B_r:=B_r(x_0)\subset B_{\frac{R}{2}}(x_0)$ with $r\in(0,1]$, $(\phi)_{B_r(x_0)}=\frac{1}{|B_r(x_0)|}\int_{B_r(x_0)}\phi(x)\d x$ and $\Sigma:=\operatorname{supp}\{\mu\}$.
\end{cor}
\begin{proof}
    From the Poinca\'re inequality in \cite[Theorem 2]{EVANS2022}, we get
    \begin{equation}\label{eq3.9}
        \fint_{B_r(x_0)}\left|\phi-(\phi)_{B_r(x_0)}\right|^p \d x \leq Cr^{p-N}\int_{B_r(x_0)} |\nabla \phi|^p \d x.
    \end{equation}
    Since $\phi$ is a truncation of the sum of $\log(u+d)$ and constant, we obtain
    \begin{align}\label{eq3.10}
        \int_{B_r(x_0)} |\nabla \phi|^p \d x \leq  \int_{B_r(x_0)} |\nabla \log(u+d)|^p \d x.
    \end{align}
    Using \eqref{eq3.9}, \eqref{eq3.10} and Logarithmic estimates of Lemma \ref{lmn3.4}, we derive
    \begin{align}\nonumber
        \fint_{B_r(x_0)}\left|\phi-(\phi)_{B_r(x_0)}\right|^p \d x \leq & Cr^{p-N} \left(\alpha \int_{B_r(x_0)}|\nabla \log(u+d)|^p \d x \right) \\ \nonumber
        \leq & C \sup\limits_{s\in\Sigma} r^{p-sp}\bigg(d^{1-p} \sup\limits_{s\in\Sigma}\left(\frac{r}{R}\right)^{sp}[\operatorname{Tail}(u_-;x_0,R)]^{p-1}+1\bigg)+C\alpha \\ \nonumber
        =& C  r^{p-\{\sup\limits_{s\in\Sigma}s\}p}\bigg(d^{1-p} \sup\limits_{s\in\Sigma}\left(\frac{r}{R}\right)^{sp}[\operatorname{Tail}(u_-;x_0,R)]^{p-1}+1\bigg)+C\alpha \\ \nonumber
        \leq & C \bigg(d^{1-p} \sup\limits_{s\in\Sigma}\left(\frac{r}{R}\right)^{sp}[\operatorname{Tail}(u_-;x_0,R)]^{p-1}+1\bigg). 
    \end{align}
    This completes the proof.
\end{proof}

\section{Local  Boundedness} \label{sec4}

In this section, we establish the following local boundedness result for weak subsolutions of \eqref{MP} using the iteration Lemma \ref{MI Lemma}. Our approach follows the method developed by Di Castro \textit{et al.} \cite[Theorem $1.1$]{DKP2016}. 
\begin{thm}\label{bddness}
    Let $u$ be a weak subsolution to the problem \eqref{MP} and $p\in(1,\infty)$. There exists a constant $C:=C(N,p,\Sigma,\mu)>0$ such that
    \begin{equation}\label{LB}
        \ess_{B_{\frac{r}{2}(x_0)}} u \leq
            \delta \operatornamewithlimits{Tail}\bigg(u_+;x_0,\frac{r}{2}\bigg)+C\delta^{-\frac{(p-1)\eta}{(\eta-1)p}}\bigg( \fint_{B_r(x_0)}u_+^p \d x \bigg)^\frac{1}{p},
    \end{equation}
    where $B_r=B_r(x_0)\subset \Omega$ with $r\in (0,1]$, $\delta\in (0,1]$, $\eta$ as in \eqref{grad sob} and $\Sigma:=\operatorname{supp}\{\mu\}$.
\end{thm}
\begin{proof}
    For $j\in \mathbb{N}\cup\{0\} $ and $r\in (0,1]$ with $B_r=B_r(x_0)\subset \Omega$, we define
    $$r_j=\frac{r}{2}\left(1+\frac{1}{2^{j}}\right),~\bar{r_j}=\frac{r_j+r_{j+1}}{2},~ B_j=B_{r_j}(x_0) \text{ and } \bar{B_j}=B_{\bar{r_j}}(x_0).$$
    Again, for $\bar{k}>0$ and $k\in \mathbb{R}$, we define
    $$k_j=k+\left( 1-\frac{1}{2^j}\right)\bar{k},~ \bar{k_j}=\frac{k_j+k_{j+1}}{2},~ \bar{\phi_j}=(u-\bar{k_j})_+ \text{ and } {\phi_j}=(u-{k_j})_+.$$
     Observe that $r_{j+1}<\bar{r_j}<r_j$ and $B_{j+1}\subset\bar{B_j}\subset B_j$. Also, note that $k_{j}<\bar{k_j}<k_{j+1}$ and $\phi_{j+1}\leq \bar{\phi_j}\leq \phi_j.$ Let $w_j\in C_c^\infty(\bar{B_j})$ be such that $0\leq w_j\leq 1$ in $\bar{B_j}$,
     $w_j=1$ in $B_{j+1}$ and $|\nabla w_j|\leq \frac{2^{j+3}}{r}$. Moreover, $k_{j+1}-\bar{k_j}=\bar{k_j}-k_j=\frac{\bar{k}}{2^{j+1}}$. Thus, using the Sobolev inequality in \eqref{sob}, we get
\begin{align}\label{eq4.2}
    \left(\frac{\bar{k}}{2^{j+2}} \right)^{\frac{(\eta-1)p}{\eta}}\left( \fint_{B_{j+1}} \phi^p_{j+1}\d x\right)^\frac{1}{\eta}&={(k_{j+1}-\bar{k_j})^{\frac{(\eta-1)p} {\eta}}} \left(\fint_{B_{j+1}}\phi^p_{j+1}w_{j}^p\d x\right)^\frac{1}{\eta} \nonumber\\
    &\leq C\left(\fint_{\bar{B_{j}}}(\bar{\phi}_{j}w_{j})^{\eta p}\d x \right)^\frac{1}{\eta} \nonumber\\
    & \leq C r^{p-N} \int_{{B_{j}}}|\nabla(\bar{\phi}_{j}w_{j})|^{ p}\d x  \nonumber\\
    & \leq C r^{p-N} \left( \int_{{B_{j}}} \bar{\phi_j}^p|\nabla w_j|^p \d x+ \int_{{B_{j}}} w_j^p |\nabla\bar{\phi_j}|^p \d x\right)\nonumber\\
    &:= J_1+J_2,
\end{align}
    for some constant $C:=C(N,p)>0$ and $\eta$ is in \eqref{grad sob}.
    
     \noindent \textbf{Estimate of $J_1$:} From the properties of $w_j$, there exists a constant $C:=C(N,p)>0$ such that
     \begin{align}\label{eq4.3}
         J_1= C r^{p-N} \int_{{B_{j}}} \bar{\phi^p_j}|\nabla w_j|^p \d x \leq C r^{p-N} \frac{2^{(j+3)p}}{r^p}\int_{{B_{j}}} \phi_j^p \d x = C 2^{jp} \fint_{{B_{j}}} \phi_j^p \d x.
     \end{align}
     \noindent  \textbf{Estimate of $J_2$:} By using the Caccioppoli inequality with tail in Corollary \ref{cor3.2} with $\phi=\bar{\phi_j}$ and $w=w_j$, we obtain
      \begin{align}\label{eq4.4}
          J_2=& C r^{p-N} \int_{{B_{j}}} w_j^p |\nabla\bar{\phi_j}|^p \d x = C r^{p-N} \alpha \int_{{B_{j}}} w_j^p |\nabla\bar{\phi_j}|^p \d x\nonumber\\ \nonumber
    \leq &  C r^{p-N}\Bigg[ \int_{(0,1)} C_{N,s,p} \left( \iint_{B_j\times B_j} { \max\{\bar{\phi_j}(x), \bar{\phi_j}(y)\}^p|w_j(x)-w_j(y)|^p}\d \nu \right)\d \mu \nonumber \\
  & + \int_{(0,1)} C_{N,s,p} \left( \ess_{x\in \supp w_j} \int_{\mathbb{R}^N\setminus B_j} \frac{\bar{\phi_j}^{p-1}(y)}{|x-y|^{N+sp}}\d y . \int_{B_j}\bar{\phi_j} w^p \d x \right)\d \mu \nonumber\\
  &+\alpha \int_{B_j} \bar{\phi_j}^p|\nabla w_j|^p \d x \Bigg]:=I_1+I_2+I_3,
    \end{align}
for some positive constant $C:=C(N,p,\Sigma,\mu)$.

\noindent \textbf{Estimate of $I_1$:} Observe that $I_1$ is symmetric. Thus without loss of generality, we assume that $u(x)\leq u(y)$, which implies $\bar{\phi_j}(x) \leq \phi_j(y)$. Therefore, there exists a positive constant $C:=C(N,p,\Sigma,\mu)$ such that
\begin{align}\label{eq4.6}
    I_1 \leq &  { C r^{p-N} \int_{(0,1)}C_{N,s,p} \frac{2^{(j+3)p}}{r^p}\left\{  \int_{B_j} \phi^p_j(y) \left( \int_{B_j} \frac{|x-y|^p}{|x-y|^{N+sp}} \d x \right)\d y \right\} \d \mu} \nonumber \\ 
    \leq & C 2^{jp} \int_{(0,1)} C_{N,s,p} \frac{r^{p-sp}}{p(1-s)} \left(\fint_{B_j} \phi^p_j(y) \d y\right) \d \mu \nonumber\\
    \leq & C 2^{jp}\left(\fint_{B_j} \phi^p_j(y) \d y\right)  \sup_{s\in \Sigma}r^{p-sp} \mu\{(0,1)\}\nonumber\\
    \leq & C 2^{jp}\left(\fint_{B_j} \phi^p_j(y) \d y\right). 
\end{align}
\noindent \textbf{Estimate of $I_2$:} For $x\in \supp w_j=\bar{B_j}$ and $y\in \mathbb{R}^N\setminus B_j$, we have
\begin{align}\nonumber
    \frac{|y-x_0|}{|y-x|}\leq \frac{|y-x|+|x-x_0|}{|y-x|}\leq 1+\frac{\bar{r_j}}{r_j-\bar{r_j}} \leq 2^{j+4}.
\end{align}
Moreover, $$\phi_j^p\geq(\bar{k_j}-k_j)^{p-1}\bar{\phi_j}.$$
Using the above two inequalities, we get
\begin{align}\label{eq4.7}
     I_2=&C r^{p-N} \int_{(0,1)} C_{N,s,p} \left( \ess_{x\in \supp w_j} \int_{\mathbb{R}^N\setminus B_j} \frac{\bar{\phi_j}^{p-1}(y)}{|y-x|^{N+sp}}\d y . \int_{B_j}\bar{\phi_j} w^p \d x \right)\d \mu \nonumber\\
     \leq & C \int_{(0,1)}  C_{N,s,p}  \bigg(r^p 2^{(j+4)(N+sp)} \int_{\mathbb{R}^N\setminus B_j} \frac{{\phi_j}^{p-1}(y)}{|y-x_0|^{N+sp}}\d y \fint_{B_j} \frac{\phi^p_j(x)}{(\bar{k_j}-k_j)^{p-1}}\d x \bigg)\d \mu \nonumber \\
     \leq & C \int_{(0,1)}C_{N,s,p} \bigg\{ r^p \frac{2^{j(N+sp)+4sp+j(p-1)}}{\bar{k}^{p-1}} \int_{\mathbb{R}^N\setminus B_{\frac{r}{2}}} \frac{{\phi_j}^{p-1}(y)}{|y-x_0|^{N+sp}}\d y \left(\fint_{B_j} \phi_j^p(x) \d x\right)\bigg\}\d \mu \nonumber\\
    \leq & C \left(\fint_{B_j} \phi_j^p(x) \d x\right) \sup_{s\in \Sigma}\{r^{p-sp}\} \sup_{s\in \Sigma}\{2^{5sp} \} \frac{\sup_{s\in \Sigma}\{2^{j(N+sp+p-1)}\}}{\bar{k}^{p-1}}  \bigg[\operatorname{Tail}\bigg(\phi_0;x_0,\frac{r}{2}\bigg)\bigg]^{p-1}  \nonumber \\
    \leq & C 2^{j(N+p-1)+\sup_{s\in \Sigma}\{spj\}} \delta^{1-p}\left(\fint_{B_j} \phi_j^p(x) \d x\right), 
\end{align}
where $C=C(N,p,\Sigma,\mu)>0$ and $\delta\operatornamewithlimits{Tail}(\phi_0;x_0,\frac{r}{2})\leq \bar{k}$ with  $\delta\in(0,1]$. 

\noindent \textbf{Estimate of $I_3$:} Using a similar arguments as in \textbf{Estimate of $J_1$}, we get
\begin{align}\label{eq4.8}
    I_3=C r^{p-N} \alpha \int_{B_j} \bar{\phi_j}^p|\nabla w_j|^p \d x \leq C 2^{jp} \fint_{{B_{j}}} \phi_j^p \d x. 
\end{align}
Applying \eqref{eq4.6}, \eqref{eq4.7} and \eqref{eq4.8} in \eqref{eq4.4}, there exists a constant $C=C(N,p,\Sigma,\mu)>0$, we obtain
\begin{align}\label{eq4.9}
    J_2 \leq C 2^{j(N+p-1)+\sup_{s\in \Sigma}\{spj\}} \delta^{1-p}\left(\fint_{B_j} \phi_j^p(x) \d x\right).
\end{align} 
Again by using \eqref{eq4.3} and \eqref{eq4.9} in \eqref{eq4.2}, there exists a positive constant $C=C(N,p,\Sigma,\mu)$, we obtain
\begin{align}\label{eq4.10}
    \left(\frac{\bar{k}}{2^{j+2}} \right)^{\frac{(\eta-1)p}{\eta}}\left( \fint_{B_j+1} \phi^p_{j+1}\d x\right)^\frac{1}{\eta} \leq C 2^{j(N+p-1)+\sup_{s\in \Sigma}\{spj\}} \delta^{1-p}\left(\fint_{B_j} \phi_j^p(x) \d x\right).
\end{align}
From \eqref{eq4.10}, we get
\begin{align}\label{eq4.11}
    \left( \fint_{B_j+1} \phi^p_{j+1}\d x\right)^\frac{1}{p} \leq C \delta^{\frac{\eta(1-p)}{p}}\left(\frac{2^{j+2}}{\bar{k}} \right)^{(\eta-1)} 2^\frac{j\eta(N+p-1)+\eta\sup_{s\in \Sigma}\{spj\}}{p} \left(\fint_{B_j} \phi_j^p(x) \d x\right)^\frac{\eta}{p}.
\end{align}
Define
\begin{equation*}
    P_j:=\left(\fint_{B_j} \phi_j^p(x) \d x\right)^{\frac{1}{p}}
    \end{equation*}
and 
\begin{equation}\label{eq4.13}
    \bar{k}=\delta \operatorname{Tail}\bigg(\phi_0;x_0,\frac{r}{2}\bigg)+C_1^{\frac{1}{\beta}}C_2^{\frac{1}{\beta^2}}\left(\fint_{B_r(x_0)} \phi_0^p(x) \d x\right)^{\frac{1}{p}},
    \end{equation}
where $C_1=C \delta^\frac{(1-p)\eta}{p}$, $C_2=2^{\eta \left\{\frac{(N+p-1)}{p} +\sup_{s\in \Sigma}s+\frac{(\eta-1)}{\eta}\right\}}$ and $\beta+1=\eta$.
Thus, using \eqref{eq4.11} and \eqref{eq4.13}, we deduce
\begin{align*}
    \frac{P_{j+1}}{\bar{k}}\leq C 2^{j\eta \left\{\frac{(N+p-1)}{p} +\sup_{s\in \Sigma}s+\frac{(\eta-1)}{\eta}\right\}}  \left( \frac{P_j}{\bar{k}}\right)^\eta
\end{align*} 
and 
\begin{align*}
    \frac{P_0}{\bar{k}}\leq C_1^{-\frac{1}{\beta}} C_2^{-\frac{1}{\beta^2}}.
\end{align*}
Therefore, using the Lemma \ref{MI Lemma}, we get
\begin{align}\nonumber
    P_j\rightarrow 0 \text{  as  } j \rightarrow \infty.
\end{align}
Hence, we have \begin{equation*}
   {\ess_{B_{\frac{r}{2}}(x_0)}} u \leq k+\bar{k}.
\end{equation*}
By setting $k=0$, then $\phi_0=u_+$, we obtain \eqref{LB}. This completes the proof.
\end{proof}

\section{Local H\"older Continuity} \label{sec5}
In this section, we establish H\"older continuity using Lemma \ref{MI Lemma}. To obtain the H\"older regularity, it is essential to impose the additional condition $\bar{s}=\inf_{s\in \Sigma}s\in (0,1)$ along with the assumption $\mu({0})=0$. To this end, we first prove the following lemma via the iteration Lemma \ref{MI Lemma}. Our approach follows that of Di Castro \textit{et al.} \cite[Lemma $5.1$]{DKP2016}. 
\begin{lem}\label{lmn5.1}
    Let $u$ be a weak solution of \eqref{MP} and $B_R(x_0)\subset \Omega$ for some $R$ with $0<r<\frac{R}{2}$, $r\in(0,1].$ Let $\kappa\in(0,\frac{1}{4}]$, we set $r_j=\kappa^j\frac{r}{2}$ and $B_j:=B_{r_j}(x_0)$ for $j\in \mathbb{N} \cup \{0\}$. Denote
    \begin{equation}\label{eq5.1}
        \frac{1}{2}\phi(r_0)=\operatorname{Tail}\bigg(u;x_0,\frac{r}{2}\bigg)+C\left(\fint_{B_r(x_0)}|u|^p \d x \right)^\frac{1}{p}, 
    \end{equation}
     and 
     \begin{equation*}
         \phi(r_j)=\left(\frac{r_j}{r_0}\right)^\sigma \phi(r_0), ~j\in \mathbb{N}, 
     \end{equation*}
     where {$\sigma \in(0,\frac{\bar{s}p}{p-1})$} with $\bar{s}=\inf_{s\in \Sigma}s$, and $C:=C(N,p,\Sigma,\mu)$ is the constant in \eqref{LB}. Then, we have
     \begin{equation}\label{eq5.3}
         \operatorname{osc}_{B_j} u=: \ess_{B_j} u -\essi_{B_j} u \leq \phi(r_j),~ j\in \mathbb{N} \cup \{0\}.
     \end{equation}
\end{lem}
\begin{proof}
By applying Lemmas \ref{lmn2.16} and \ref{lmn2.17}, we deduce that both functions $u_+$ and $(-u)+=u-$ are weak subsolutions. Then, by Theorem \ref{bddness}, the estimate \eqref{eq5.3} holds for $j=0$. We now prove \eqref{eq5.3} by induction. Assume that \eqref{eq5.3} holds for all $i=0,1,2,\ldots,j$ for some $j\geq 0$. To complete the induction argument, it suffices to establish \eqref{eq5.3} for $i=j+1$. We note that either
\begin{align}\label{eq5.4}
    \frac{|\{u\geq \essi_{B_j}u+\frac{\phi(r_j)}{2}\} \cap 2B_{j+1}|}{| 2B_{j+1}|}\geq \frac{1}{2},
\end{align}
or
\begin{align}\label{eq5.5}
    \frac{|\{u\leq \essi_{B_j}u+\frac{\phi(r_j)}{2}\} \cap 2B_{j+1}|}{| 2B_{j+1}|}\geq \frac{1}{2},
\end{align}
hold. Let us define
\begin{align}
    u_j=\begin{cases}
        & u- \essi_{B_j} u,  ~~~~~~~~~~~~~\text{ when } \eqref{eq5.4} \text{ holds },\\
        &\phi(r_j)-(u- \essi_{B_j} u), \text{ when } \eqref{eq5.5} \text{ holds}.\label{eq5.6}
    \end{cases}
\end{align}
Therefore, $u_j$ is a weak solution such that
\begin{enumerate}
    \item[(i)]  $u_j\geq 0$ in $B_j$,
    \item[(ii)]  $ \frac{|\{u_j\geq \frac{\phi(r_j)}{2}\} \cap 2B_{j+1}|}{| 2B_{j+1}|}\geq \frac{1}{2},$ \text { and}
    \item[(iii)] {$\ess_{B_i}|u_j|\leq 2 \phi(r_i)$  $\forall ~i\in  \{0,1,2,...,j\}.$} 
\end{enumerate}
Before proving the main result, we first claim that 
\begin{align}\label{eq5.7}
    [\operatorname{Tail}(u_j;x_0,r_j)]^{p-1} \leq C \kappa^{-\sigma(p-1)}\phi(r_j)^{p-1},
\end{align}
where $C:=C(N,p,\mu,\Sigma,|\frac{\bar{s}p}{p-1}-\sigma|)>0$ is a constant. 

\noindent\textbf{Proof of the claim:} Using the condition (iii) above, \eqref{eq5.9} and \eqref{eq5.100}, we get
\begin{align}\label{eq5.8}
   [\operatorname{Tail}(u_j;x_0,r_j)]^{p-1}=&\int_{(0,1)}{C_{N,p,s}}r_j^{sp}\left(\int_{\mathbb{R}^N\setminus B_{r_j}(x_0)}\frac{|u_j(x)|^{p-1}}{|x-x_0|^{N+sp}}\d x\right)\d \mu(s) \nonumber\\
   =& \int_{(0,1)}C_{N,p,s} r_j^{sp} \bigg( \sum_{i=1}^j \int_{B_{i-1}\setminus B_i} \frac{|u_j(x)|^{p-1}}{|x-x_0|^{N+sp}}\d x \nonumber \\
   &+ \int_{\mathbb{R}^N\setminus B_0} \frac{|u_j(x)|^{p-1}}{|x-x_0|^{N+sp}}\d x \bigg)\d \mu(s)  := L_1+L_2 \nonumber \\
      \leq & C {\phi(r_{j})^{p-1} \kappa^{-\sigma(p-1)}  \sum_{i=1}^j \kappa^{(j-i)\{\inf_{s\in \Sigma}[s]p-\sigma(p-1)\}} }  \nonumber \\
     \leq & C{\phi(r_{j})^{p-1} \kappa^{-\sigma(p-1)}  \sum_{i=1}^j \kappa^{(j-i)\{\bar{s}p-\sigma(p-1)\}} }  \nonumber \\
    \leq  & C  \phi(r_{j})^{p-1} \frac{\kappa^{-\sigma(p-1)}}{1-\kappa^{\bar{s}p-\sigma(p-1)}} \nonumber \\ 
     \leq  & C \frac{4^{\bar{s}p-\sigma(p-1)}}{\log4 \{\bar{s}p-\sigma(p-1)\}} \kappa^{-\sigma(p-1)} \phi(r_{j})^{p-1}= C \kappa^{-\sigma(p-1)} \phi(r_{j})^{p-1},
\end{align}
where $C=C(N,p,\mu,\Sigma,|\frac{\bar{s}p}{p-1}-\sigma|)>0$ is a constant and $0<\kappa\leq \frac{1}{4}$, $0<\sigma<\frac{\bar{s}p}{p-1}$. By using \eqref{eq5.1}, we get
\begin{align}\label{eq5.9}
   L_1=& \int_{(0,1)}C_{N,p,s} r_j^{sp} \left( \sum_{i=1}^j \int_{B_{i-1}\setminus B_i} \frac{|u_j(x)|^{p-1}}{|x-x_0|^{N+sp}}\d x \right)\d \mu(s) \nonumber\\
   \leq & \int_{(0,1)}C_{N,p,s} r_j^{sp} \left( \sum_{i=1}^j\ess_{B_{i-1}}|u_j|^{p-1} {{\int_{\mathbb{R}^N\setminus B_i}} \frac{1}{|x-x_0|^{N+sp}}\d x} \right) \d \mu(s)\nonumber \\
    \leq & C { \int_{(0,1)} {C_{N,p,s}\frac{1}{s}} \sum_{i=1}^j\left(\frac{r_j}{r_i}\right)^{sp}\phi(r_{i-1})^{p-1}d \mu(s)}\nonumber \\
    = & C  \int_{(0,1)} C_{N,p,s}\frac{1}{s}\phi(r_{0})^{p-1}\left( \frac{r_j}{r_0}\right)^{\sigma(p-1)}  \sum_{i=1}^j\left(\frac{r_{i-1}}{r_i}\right)^{\sigma(p-1)}\left( \frac{r_j}{r_i}\right)^{sp-\sigma(p-1)}\d \mu(s)\nonumber \\
    = & C  \int_{(0,1)} C_{N,p,s}\frac{1}{s} \phi(r_{j})^{p-1}  \sum_{i=1}^j \kappa^{-\sigma(p-1)} \kappa^{(j-i)\{sp-\sigma(p-1)\}}  \d  \mu(s)  \nonumber \\ 
     \leq & C \phi(r_{j})^{p-1} \kappa^{-\sigma(p-1)}  \sum_{i=1}^j \sup_{s\in \Sigma} \kappa^{(j-i)\{sp-\sigma(p-1)\}} \mu\{(0,1)\}, \nonumber \\
     \leq & C \phi(r_{j})^{p-1} \kappa^{-\sigma(p-1)}  \sum_{i=1}^j \sup_{s\in \Sigma} \kappa^{(j-i)\{sp-\sigma(p-1)\}}
\end{align}
and
\begin{align}\label{eq5.100}
    L_2=& \int_{(0,1)} C_{N,p,s}r_j^{sp} \left( \int_{\mathbb{R}^N\setminus B_0} \frac{|u_j(x)|^{p-1}}{|x-x_0|^{N+sp}}\d x \right)\d \mu(s)\nonumber\\
    \leq & C\int_{(0,1)}C_{N,p,s} r_j^{sp} \left( \int_{\mathbb{R}^N\setminus B_0} \frac{|u(x)|^{p-1}+\phi^{p-1}(r_0)+\sup_{B_0}|u|^{p-1}}{|x-x_0|^{N+sp}}\d x \right)\d \mu(s)\nonumber\\
    \leq & C\int_{(0,1)}C_{N,p,s} \left(\frac{r_j}{r_0}\right)^{sp}\left[ {\frac{1}{s}}\{\phi^{p-1}(r_0)+\sup_{B_0}|u|^{p-1}\} + r_0^{sp} \int_{\mathbb{R}^N\setminus B_0} \frac{|u(x)|^{p-1}}{|x-x_0|^{N+sp}}\d x\right] \d \mu(s)\nonumber\\
     \leq & C\int_{(0,1)}C_{N,p,s} \left(\frac{r_j}{r_1}\right)^{sp}\left[{\frac{1}{s}}\phi^{p-1}(r_0) + r_0^{sp} \int_{\mathbb{R}^N\setminus B_0} \frac{|u(x)|^{p-1}}{|x-x_0|^{N+sp}}\d x\right] \d \mu(s)\nonumber\\
 = & C\int_{(0,1)} C_{N,p,s}\left(\frac{r_j}{r_0}\right)^{\sigma(p-1)} \kappa^{(j-1)\{sp-\sigma(p-1)\}} \kappa^{-\sigma(p-1)} \bigg[{\frac{1}{s}}\phi^{p-1}(r_0) \nonumber \\
& + r_0^{sp} \int_{\mathbb{R}^N\setminus B_0} \frac{|u(x)|^{p-1}}{|x-x_0|^{N+sp}}\d x\bigg] \d \mu(s)\nonumber\\
     \leq & C \left(\frac{r_j}{r_0}\right)^{\sigma(p-1)} \sup_{s\in \Sigma} \kappa^{(j-1)\{sp-\sigma(p-1)\}} \kappa^{-\sigma(p-1)} \left[\phi^{p-1}(r_0) + [\operatorname{Tail}(u;x_0,r_0)]^{p-1} \right] \mu\{(0,1)\} \nonumber \\
     \leq & C \left(\frac{r_j}{r_0}\right)^{\sigma(p-1)} \phi^{p-1}(r_0) \sup_{s\in \Sigma} \kappa^{(j-1)\{sp-\sigma(p-1)\}} \kappa^{-\sigma(p-1)}\nonumber \\
     = & C \phi(r_j)^{p-1} \kappa^{-\sigma(p-1)}\sup_{s\in \Sigma} \kappa^{(j-1)\{sp-\sigma(p-1)\}}. 
\end{align}
Therefore, \eqref{eq5.8}, \eqref{eq5.9} and \eqref{eq5.100} proves the estimate \eqref{eq5.7}, concluding the claim.

Next, we proceed to prove the main result, which is divided into two steps.

\noindent \textbf{Step 1:} In this step, we prove that
\begin{align}\label{eq5.10}
    \frac{|\{u_j \leq 2 \tau\phi(r_j)\}\cap 2B_{r_{j+1}}(x_0)|}{|2B_{j+1}(x_0)|} \leq \frac{\bar{C}}{\log\frac{1}{\kappa}},
\end{align}
where $\tau=\kappa^{\frac{\bar{s}p}{p-1}-\sigma}$ for some positive constant $\bar{C}:=\bar{C}(N,p,\Sigma,\mu,|\frac{\bar{s}p}{p-1}-\sigma|)$. We denote
\begin{align*}
    \epsilon:=\log\left( \frac{\frac{\phi(r_j)}{2}+\tau\phi(r_j)}{3\tau \phi(r_j)}\right)=\log\left(\frac{\frac{1}{2}+\tau}{3\tau}\right) \sim \log \left( \frac{1}{\tau} \right), \text{ as } \tau \rightarrow 0,
\end{align*}
and 
\begin{align*}
    q:=\min\left\{ \left( \log\left( \frac{\frac{\phi(r_j)}{2}+\tau\phi(r_j)}{u_j+\tau \phi(r_j)}\right) \right)_+, \epsilon  \right\}.
\end{align*}
Using the condition $(ii)$ of $u_j$, we get
\begin{align}\label{eq5.13}
    \epsilon= &\frac{1}{|\{u_j\geq \frac{\phi(r_j)}{2}\} \cap 2B_{j+1}|}\int_{\{u_j\geq \frac{\phi(r_j)}{2}\} \cap 2B_{j+1}} \epsilon \d x \nonumber \\
     = &\frac{1}{|\{u_j\geq \frac{\phi(r_j)}{2}\} \cap 2B_{j+1}|}\int_{\{q=0\} \cap 2B_{j+1}} \epsilon \d x \nonumber \\
     \leq &\frac{2}{| 2B_{j+1}|}\int_{ 2B_{j+1}}\left(\epsilon-q\right) \d x=2\left(\epsilon-(q)_{ 2B_{j+1}}\right),
\end{align}
where $(q)_{ 2B_{j+1}}=\fint_{{ 2B_{j+1}}}q \d x$. Integrating both sides of \eqref{eq5.13} over $[{\{q=\epsilon\} \cap 2B_{j+1}}]$, we obtain
\begin{align}\label{eq5.14}
    \frac{|{\{q=\epsilon\} \cap 2B_{j+1}}|\epsilon}{|2B_{j+1}|} \leq & \frac{2}{|2B_{j+1}|} \int_{[{\{q=\epsilon\} \cap 2B_{j+1}}]} \left(\epsilon-(q)_{ 2B_{j+1}}\right) \d x \nonumber \\
    \leq & \frac{2}{|2B_{j+1}|} \int_{\{ 2B_{j+1}\}} |q-(q)_{ 2B_{j+1}}| \d x. 
\end{align}
Using Corollary \ref{coro} along with $b=\frac{\phi(r_j)}{2}$, $d=\tau \phi(r_j)$, $a=e^\epsilon$, and {$B_{j+1}\subset  B_{\frac{j}{2}}$}, we conclude that there exists a constant $C>0$, such that
\begin{align}\label{eq5.16}
    \fint_{\{ 2B_{j+1}\}} |q-(q)_{ 2B_{j+1}}|^p \d x \leq & C  \bigg((\tau \phi(r_j))^{1-p} \sup\limits_{s\in\Sigma}\left(\frac{r_{j+1}}{r_j}\right)^{sp}[\operatorname{Tail}((u_j)_-;x_0,2r_j)]^{p-1}+1\bigg), \nonumber \\
    \leq & C  \bigg((\tau \phi(r_j))^{1-p} \left(\frac{r_{j+1}}{r_j}\right)^{\bar{s}p}[\operatorname{Tail}(u_j;x_0,r_j)]^{p-1}+1\bigg).
\end{align}
From \eqref{eq5.7} and \eqref{eq5.16}, we get
\begin{align}\label{eq5.17}
    \fint_{\{ 2B_{j+1}\}} |q-(q)_{ 2B_{j+1}}| \d x \leq & C \left(\tau^{1-p}\kappa^{\bar{s}p-\sigma(p-1)}+1\right)\leq C.
\end{align}
Now using \eqref{eq5.14} and \eqref{eq5.17}, we obtain
\begin{align*}
    \frac{|\{u_j \leq 2 \tau\phi(r_j)\}\cap 2B_{r_{j+1}}(x_0)|}{|2B_{j+1}(x_0)|}=\frac{|{\{q=\epsilon\} \cap 2B_{j+1}}|}{|2B_{j+1}|}\leq \frac{C}{\epsilon} \leq \frac{\bar{C}}{\log\frac{1}{\kappa}}.
\end{align*}
\textbf{Step 2:} Finally, we are going to use an iterative argument to show \eqref{eq5.3} for $i=j+1$. Initially for any $i\in \mathbb{N}\cup \{0\}$, we denote $\gamma_i=\left(1+\frac{1}{2^i}\right)r_{j+1}$, $\bar{\gamma_i}=\frac{\gamma_i+\gamma_{i+1}}{2}$, $B^i=B_{\gamma_i}(x_0)$, $\bar{B^i}=B_{\bar{\gamma_i}}(x_0)$, $\rho_i=\left(1+\frac{1}{2^i}\right)\tau \phi(r_j)$ and $D^i=B^i \cap \{u_j \leq \rho_i\}$. Note that $r_{j+1}\leq \gamma_{i+1}\leq \bar{\gamma_i}\leq \gamma_i \leq 2 r_{j+1}<r_j$. 
\par Let us consider cut-off functions $w_i\in C_c^\infty(\bar{B^i})$ such that $0\leq w_i\leq 1$ in $\bar{B^i}$, $w_i=1$ in $B^{i+1}$ and $|\nabla w_i|\leq \frac{C2^{i}}{\gamma_i}$ in $\bar{B^i}$ with $C:=C(N,p)>0$. Let $\phi_i=(\rho_i-u_j)_+$. Note that $\rho_i-\rho_{i+1}=\frac{\tau \phi(r_j)}{2^{i+1}}$ and $\phi_i \leq \rho_i\leq 2\tau \phi(r_j)$ in $B^i$. Applying the Sobolev inequality \eqref{sob} and the convexity property, we obtain
     \begin{align}\label{eq5.19}
         (\rho_i-\rho_{i+1})^p\left( \frac{|D^{i+1}|}{|B^{i+1}|}\right)^\frac{1}{\eta} =& \frac{1}{|B^{i+1}|^\frac{1}{\eta}} \left( \int_{B^{i+1} \cap \{u_j \leq \rho_{i+1}\}} (\rho_i-\rho_{i+1})^{p\eta} \d x \right)^\frac{1}{\eta}\nonumber\\
         \leq & \left( \fint_{B^{i+1}}\phi_i^{\eta p} \d x\right)^\frac{1}{\eta}=\left( \fint_{B^{i+1}}\phi_i^{\eta p} w_i^{\eta p}\d x\right)^\frac{1}{\eta} \nonumber \\
          \leq & C \left( \fint_{B^{i}}\phi_i^{\eta p} w_i^{\eta p}\d x\right)^\frac{1}{\eta} \nonumber \\
           \leq & C \gamma_i^p \fint_{B^{i}} |\nabla(\phi_i w_i)|^p \d x  \nonumber \\
          \leq & C r_{j+1}^p  \left( \fint_{{B^{i}}} {\phi_i}^p|\nabla w_i|^p \d x+ \fint_{{B^{i}}} w_i^p |\nabla{\phi_i}|^p \d x\right)\nonumber\\
    := &S_1+S_2,
     \end{align}
 where $C:=C(N,p)>0$ is a constant and $\eta$ is as in \eqref{grad sob}.

 \noindent  \textbf{Estimate of $S_1$:} Using the properties of $w_j$, there exists a constant $C:=C(N,p)>0$ such that
     \begin{align}\label{eq5.20}
         S_1=& C r^{p}_{j+1} \fint_{{B^{i}}} {\phi^p_i}|\nabla w_i|^p \d x=\frac{C r^{p}_{j+1}}{|B^i|} \int_{{B_{i}\cap \{u_j \leq \rho_i\}}} {\phi^p_i}|\nabla w_i|^p \d x \nonumber \\
         \leq &C r^{p}_{j+1}\frac{2^{ip}}{\gamma_i^p} (2\tau \phi(r_j))^p \frac{|D^i|}{|B^i|}\leq C 2^{ip} (\tau \phi(r_j))^p \frac{|D^i|}{|B^i|}.
  \end{align}
   \noindent\textbf{Estimate of $S_2$:} Using the Caccioppoli inequality with tail Theorem \ref{lmn3.1} with $\phi={\phi_i}$ and $w=w_i$, there exists a constant $C:=C(p)>0$ such that
      \begin{align}\label{eq5.21}
          S_2=& Cr^{p}_{j+1} \fint_{{B^{i}}} w_i^p |\nabla{\phi_i}|^p \d x = \frac{C r^{p}_{j+1} \alpha}{|B^i|} \int_{{B^ {i}}} w_i^p |\nabla{\phi_i}|^p \d x\nonumber\\ \nonumber
          \leq &  \frac{C r^{p}_{j+1} }{|B^i|} \left( \int_{(0,1)} C_{N,s,p} \left( \iint_{B^i\times B^i} {|{\phi}_i(x) w_i(x)-{\phi_i}(y)w_i(y)|^p}\d \nu\right) \d \mu+\alpha \int_{{B^i}} w_i^p |\nabla{\phi_i}|^p \d x  \right)\\ \nonumber
    \leq & \frac{C r^{p}_{j+1} }{|B^i|}\Bigg[ \int_{(0,1)} C_{N,s,p} \left( \iint_{B^i\times B^i} { \max\{{\phi_i}(x), {\phi_i}(y)\}^p|w_i(x)-w_i(y)|^p}\d \nu\right)\d \mu \nonumber \\
  & + \int_{(0,1)} C_{N,s,p} \left( \ess_{x\in \supp w_i} \int_{\mathbb{R}^N\setminus B^i} \frac{{\phi_i}^{p-1}(y)}{|x-y|^{N+sp}}\d y . \int_{B^i}{\phi_i} w^p \d x \right)\d \mu \nonumber\\
  &+\alpha \int_{B^i} {\phi_i}^p|\nabla w_i|^p \d x \Bigg]:=Q_1+Q_2+Q_3.
    \end{align}
\noindent\textbf{Estimate of $Q_1$:} By using the property of $w_i$, there exists $C:=C(N,p,\Sigma,\mu)$ such that
\begin{align}\label{eq5.22}
    Q_1=& \frac{C r^{p}_{j+1} }{|B^i|} \int_{(0,1)} C_{N,s,p} \bigg( \iint_{B^i\times B^i} \frac{ \max\{{\phi_i}(x), {\phi_i}(y)\}^p|w_i(x)-w_i(y)|^p}{|x-y|^{N+sp}}\d x\d y\bigg)\d \mu \nonumber \\
    \leq & \frac{C2^{ip} }{|B^i|} \rho_i^p \int_{(0,1)} C_{N,s,p} {\left\{\int_{D^i} \left( \int_{B^i}\frac{1}{|x-y|^{N-p+sp}}\d y\right) \d x \right\}\d \mu} \nonumber \\
    \leq & \frac{C2^{ip} }{|B^i|} (\tau \phi(r_j))^p  \sup_{s\in \Sigma}\gamma_i^{(p-sp)} |D_i| \mu\{(0,1)\}\nonumber \\
     \leq & C 2^{ip} (\tau \phi(r_j))^p \frac{|D^i|}{|B^i|},
\end{align}
\noindent\textbf{Estimate of $Q_2$:} Also, we have
\begin{align}\label{eq5.23}
    Q_2=\frac{C r^{p}_{j+1} }{|B^i|} \int_{(0,1)} C_{N,s,p} \left( \ess_{x\in \supp w_i} \int_{\mathbb{R}^N\setminus B^i} \frac{{\phi_i}^{p-1}(y)}{|x-y|^{N+sp}}\d y . \int_{B^i}{\phi_i} w^p \d x \right)\d \mu. 
\end{align}
Observe that for $x\in \supp w_i=\bar{B^i}$ and $y\in \mathbb{R}^N\setminus B^i$, we have
\begin{align}\label{eq5.24}
    \frac{1}{|y-x|}=\frac{|y-x_0|}{|y-x|} \frac{1}{|y-x_0|}\leq \frac{1}{|y-x_0|} \left( 1+\frac{\bar{\gamma_i}}{\gamma_i-\bar{\gamma_i}}\right)\leq \frac{1}{|y-x_0|}  2^{i+3}
\end{align}
 and
 \begin{align}\label{eq5.25}
     \int_{B^i}{\phi_i} w^p \d x\leq C(\tau\phi(r_j))|D^i|,
 \end{align}
for some constant $C:=C(N,p)>0$. 
Now  using \eqref{eq5.8}, we obtain
 \begin{align}\label{eq5.26}
     [\operatorname{Tail}(\phi_i;x_0,r_{j+1})]^{p-1}\leq & C \bigg[\int_{(0,1)} C_{N,s,p}\left(\int_{B_j\setminus B_{j+1}} \frac{r^{sp}_{j+1}\phi_i^{p-1}}{|x-x_0|^{N+sp}}\d x\right)\d \mu \nonumber \\
     &+ \sup_{s\in \Sigma} \left( \frac{r_{j+1}}{r_j}\right)^{sp} [\operatorname{Tail}(\phi_i;x_0,r_{j})]^{p-1}\bigg] \nonumber \\
      \leq &  C\bigg[\int_{(0,1)}C_{N,s,p}{\left(\int_{B_j\setminus B_{j+1}} \frac{r^{sp}_{j+1}(\tau \phi(r_j))^{p-1}}{|x-x_0|^{N+sp}}\d x\right)\d \mu  }\nonumber \\
      &+ \sup_{s\in \Sigma} \kappa^{sp} [\operatorname{Tail}(u_j;x_0,r_{j})]^{p-1}\bigg] \nonumber \\
    \leq &  C \left((\tau \phi(r_j))^{p-1} + \kappa^{\inf_{{s}\in \Sigma}\{sp\} } \kappa^{-\sigma(p-1)} \phi^{p-1}(r_{j}) \right) \nonumber  \\
     \leq &  C \left( 1+\frac{k^{\bar{s}p-\sigma(p-1)}}{\tau^{p-1}} \right)(\tau \phi(r_j))^{p-1} \leq C (\tau \phi(r_j))^{p-1},
 \end{align}
for some constant $C:=C(N,p,\mu, \Sigma)>0$.
Using \eqref{eq5.24}, \eqref{eq5.25} and \eqref{eq5.26} in \eqref{eq5.23}, we get
 \begin{align}\label{eq5.27}
     Q_2 \leq & \frac{C }{|B^i|}(\tau\phi(r_j))|D^i| \sup_{s\in \Sigma}\{2^{i(N+sp)}\} \sup_{s\in \Sigma}\{ r^{p-sp}_{j+1}\} [\operatorname{Tail}(\phi_i;x_0,r_{j+1})]^{p-1} \nonumber \\
      \leq &C \frac{|D^i| }{|B^i|} 2^{i(N+p\sup_{s\in \Sigma}s)} (\tau\phi(r_j))^p.
 \end{align}
 \noindent\textbf{Estimate of $Q_3$:} Following an argument similar to the \textbf{Estimate of $S_1$} in \eqref{eq5.20}, we obtain
\begin{align}\label{eq5.28}
    Q_3=\frac{C r^{p}_{j+1} }{|B^i|} \alpha \int_{B^i} {\phi_i}^p|\nabla w_i|^p \d x \leq C 2^{ip} (\tau \phi(r_j))^p \frac{|D^i|}{|B^i|}. 
\end{align}
Now using \eqref{eq5.22}, \eqref{eq5.27} and \eqref{eq5.28} in \eqref{eq5.21}, we obtain
\begin{align}\label{eq5.29}
    S_2 \leq C2^{i(N+p+p\sup_{s\in \Sigma}s)} (\tau\phi(r_j))^p\frac{|D^i| }{|B^i|}.
\end{align}
Recall that $(\rho_i-\rho_{i+1})=\frac{\tau \phi(r_j)}{2^{i+1}}$. Thus, applying \eqref{eq5.20} and \eqref{eq5.29} in \eqref{eq5.19},
we get
\begin{align}\label{eq5.30}
    \left(\frac{\tau \phi(r_j)}{2^{i+1}}\right)^p\left( \frac{|D^{i+1}|}{|B^{i+1}|}\right)^\frac{1}{\eta}\leq C2^{i(N+p+p\sup_{s\in \Sigma}s)} (\tau\phi(r_j))^p\frac{|D^i| }{|B^i|}.
\end{align}
From \eqref{eq5.30}, we get 
\begin{align}\label{eq5.31}
    \frac{|D^{i+1}|}{|B^{i+1}|} \leq C 2^{i(N+2p+p\sup_{s\in \Sigma}s)\eta}\left(\frac{|D^i| }{|B^i|}\right)^{\eta}.
\end{align}
Let $\frac{|D^i| }{|B^i|}=P_i$, $C_1=C$, $C_2=2^{(N+2p+p\sup_{s\in \Sigma}s)\eta}$, $\beta+1=\eta$ and $\kappa_1=C_1^{-\frac{1}{\beta}}C_2^{-\frac{1}{\beta^2}}$. Therefore, using \eqref{eq5.10} and \eqref{eq5.31}, we obtain
$$P_{i+1}\leq C_1 C_2^i P_i^{\beta+1}~~\text{ and }~~P_0 \leq \frac{\bar{C}}{\log\frac{1}{\kappa}}.$$
Choosing, $\kappa \leq \frac{1}{2}\min\{\frac{1}{4},e^{-\frac{\bar{C}}{\kappa_1}}\}$, we get
$$P_0\leq \kappa_1.$$
Therefore, using Lemma \ref{MI Lemma}, we conclude that $$\lim_{i\rightarrow \infty}P_i=0.$$
Thus, we have $u_j\geq\tau\phi(r_j)$ in $B_{j+1}$. Moreover, using the definition of $u_j$, we deduce that
\begin{align*}
   \operatorname{osc}_{B_{j+1}} u\leq (1-\tau) \phi(r_j) =(1-\tau)\kappa^{-\sigma}\phi(r_{j+1}) \leq \phi(r_{j+1}), 
\end{align*}
where $\sigma\in(0, \frac{\bar{s}p}{p-1})$ is sufficiently small such that $\kappa^\sigma \geq (1-\tau)$. Therefore, the induction principle is true for $i=j+1$. This completes the proof.
\end{proof}
Using the above Lemma \ref{lmn5.1}, we conclude the H\"older continuity as follow.
\begin{thm}\label{Holder t}
    Let $u$ be a weak solution of \eqref{MP}. Then $u$ is locally H\"older continuous in $\Omega$. Moreover, there exists $\sigma\in(0, \frac{\bar{s}p}{p-1})$ and positive constant $C:=C(N,p,\Sigma,\mu)$ such that
    \begin{equation*}
        \operatorname{osc}_{B_\epsilon(x_0)}u\leq C\left(\frac{\epsilon}{r} \right)^\sigma \left(  \operatorname{Tail}(u;x_0,{r})+C\left(\fint_{B_{2r}(x_0)}|u|^p \d x \right)^\frac{1}{p} \right), 
    \end{equation*}
    where $B_{2r}(x_0)\subset \Omega$ such that $r\in (0,1]$ and $\epsilon\in (0,r]$.
\end{thm}


\section{Weak Harnack Inequality} \label{sec6}
This section is devoted to proving the weak Harnack inequality for the problem \eqref{MP}. The following lemma establishes the expansion-of-positivity technique, which works well with superposition operators.
\begin{lem}\label{lmn6.1}
    Let $u$ be a weak supersolution of \eqref{MP} such that $u\geq 0$ in $B_R(x_0)\subset \Omega$. Suppose that $\zeta\geq 0$ and there exists $\sigma\in (0,1]$ such that 
    \begin{align}\label{eq6.1}
        |\{u\geq \zeta\} \cap B_r(x_0)| \geq \sigma |B_r(x_0)|,
    \end{align}
    for some $r\in( 0,1]$ with $0<16r<R$. Then there exists a constant $C:=C(N,p,\mu,\Sigma)>0$ such that
    \begin{align}\label{eq6.2}
        |B_{6r}(x_0)\cap \{u\leq 2\delta \zeta-{ \frac{1}{2}\sup_{s\in \Sigma}\bigg(\frac{r}{R}\bigg)^{\frac{sp}{p-1}}  \operatorname{Tail}\big(u_-;x_0,R\big)-\epsilon \} }| \leq \frac{C|B_{6r}(x_0)|}{\sigma \log \frac{1}{2\delta}},
    \end{align}
    for any $\delta \in (0,\frac{1}{4})$ and $\epsilon>0$.
\end{lem}
\begin{proof}
    Let $w\in C_c^\infty(B_{7r}(x_0))$ be such that $w(x)=1$ for all $x\in B_{6r}(x_0)$, $0\leq w(x)\leq 1$ for all $x\in B_{7r}(x_0)$ and $|\nabla w|\leq \frac{8}{r}$. We choose $v=u+d_\epsilon$, where
    \begin{align*}
        d_\epsilon=\frac{1}{2}\sup_{s\in \Sigma}\bigg(\frac{r}{R}\bigg)^{\frac{sp}{p-1}}  \operatorname{Tail}\big(u_-;x_0,R\big)+\epsilon.
    \end{align*}
    Since, $v$ is a weak supersolution of \eqref{MP}, choosing $\psi=v^{1-p}w^p$ as a test function to get
    \begin{align}\label{eq6.4}
        0\leq & \int_{(0,1)}C_{N, s, p}\left(\iint_{B_{8r}(x_0)\times B_{8r}(x_0)} {{ \mathcal{A}(v(x,y))(v^{1-p}(x)w^p(x)-v^{1-p}(y)w^p(y))}}\d \nu \right) \d \mu(s)\nonumber\\
        &+2 \int_{(0,1)}C_{N, s, p} \left(\int_{\mathbb{R}^N\setminus {B_{8r}(x_0)}}\int_{B_{8r}(x_0)} {{ \mathcal{A}(v(x,y))(v^{1-p}(x)w^p(x))}}\d \nu\right) \d \mu(s)\nonumber\\
        & + \alpha \int_{B_{8r}(x_0)} |\nabla v(x)|^{p-2}\nabla v(x) \cdot \nabla (v^{1-p}(x)w^p(x)) \d x:= L_1+2L_2+\alpha L_3.
    \end{align}
        \textbf{Estimate of $L_1$:} Using similar arguments as in \cite[Lemma $5.1$]{AGKR2025i} (specifically $(5.23)$ and $(5.28)$), we obtain
        \begin{align}\label{eq6.5}
            L_1 \leq& -\frac{1}{C} \int_{(0,1)} C_{N, s, p}\bigg( \iint_{B_{8r}(x_0)\times B_{8r}(x_0)}\bigg|\log\bigg(\frac{v(x)}{v(y)} \bigg)\bigg|^p {w^p(y)}\d \nu \bigg) \d \mu(s)+C\sup_{s\in \Sigma} r^{N-sp} \nonumber \\
            \leq &  -\frac{1}{C} \int_{(0,1)} C_{N, s, p} \bigg(\iint_{B_{6r}(x_0)\times B_{6r}(x_0)}\bigg|\log\bigg(\frac{v(x)}{v(y)}\bigg) \bigg|^p\d \nu \bigg) \d \mu(s)+Cr^{N-p}.
        \end{align}
\textbf{Estimate of $L_2$:} Also, we have
            \begin{align}\label{eq6.6}
                L_2= &\int_{(0,1)}C_{N, s, p}\left(\int_{\mathbb{R}^N\setminus {B_{8r}(x_0)} \cap \{v(y)<0\}}\int_{B_{8r}(x_0)} {{ \mathcal{A}(v(x,y))(v^{1-p}(x)w^p(x))}}\d \nu \right) \d \mu(s)\nonumber\\
                &+ \int_{(0,1)}C_{N, s, p}\left(\int_{\mathbb{R}^N\setminus {B_{8r}(x_0)}\cap \{v(y)\geq 0\}}\int_{B_{8r}(x_0)}{{ \mathcal{A}(v(x,y))(v^{1-p}(x)w^p(x))}} \d \nu \right) \d \mu(s)\nonumber\\
                :=&L_2'+L_2''.
            \end{align}
            \textbf{Estimate of $L_2'$:} Since the support of $w$ is in $B_{7r}(x_0)$, using the definition of $v$ and $v_-\leq u_-$, we get
            \begin{align}\label{eq6.7}
                L_2'=&\int_{(0,1)}C_{N, s, p}\left(\int_{\mathbb{R}^N\setminus {B_{8r}(x_0)} \cap \{v(y)<0\}}\int_{B_{8r}(x_0)} \frac{{|v(x)+v_-(y)|^{p-1}}{(v^{1-p}(x)w^p(x))}}{|x-y|^{N+s p}} \d x \d y\right) \d \mu(s)\nonumber\\
                \leq & C  r^N \int_{(0,1)}C_{N, s, p}\left( \int_{\mathbb{R}^N\setminus {B_{8r}(x_0)}} \bigg(1+\frac{u_-(y)}{d_\epsilon}\bigg)^{p-1} {\frac{8^{N+sp}}{|x_0-y|^{N+s p}}} \right) \d \mu(s) \nonumber\\
                 \leq & C r^N \sup_{s\in\Sigma} r^{-sp} + C r^N \sup_{s\in \Sigma}8^{N+sp} \sup_{s\in \Sigma} \bigg(\frac{r}{R}\bigg)^{sp}  \sup_{s\in \Sigma}r^{-sp} d_\epsilon^{1-p} [\operatorname{Tail}(u_-,x_0,R)]^{p-1} \nonumber\\
                 \leq& C r^{N-p}.
            \end{align}
\textbf{Estimate of $L_2''$:} Note that $v(x)-v(y) \leq 0$ implies $L_2'' \leq 0$ and $v(x)-v(y) \geq 0$ implies $|v(x)-v(y)|^{p-2}(v(x)-v(y)) \leq v^{p-1}(x)$. Therefore,  using the support of $w$ in $B_{7r}(x_0)$, we obtain 
\begin{align}\label{eq6.8}
    L_2'' \leq &\int_{(0,1)}C_{N, s, p}\left(\int_{\mathbb{R}^N\setminus {B_{8r}(x_0)} \cap \{v(y)\geq 0\}}\int_{B_{8r}(x_0)} \frac{w^p(x)}{|x-y|^{N+s p}}\d x \d y \right) \d \mu(s) \nonumber \\
    \leq& \int_{(0,1)}C_{N, s, p}\left(\int_{\mathbb{R}^N\setminus {B_{8r}(x_0)}} \int_{B_{7r}(x_0)} {\frac{8^{N+sp}}{|x_0-y|^{N+s p}}} \d x \d y \right) \d \mu(s) \nonumber \\
    \leq&C  \sup_{s\in\Sigma} r^{N-sp} \leq C r^{N-p}.
\end{align}
\textbf{Estimate of $L_3$:} Since, {$|\nabla \log v|=|v|^{-1} |\nabla v|$}, using Young's inequality and property of $w$, we get
\begin{align}\label{eq6.9}
    L_3=&\int_{B_{8r}(x_0)} |\nabla v(x)|^{p-2}\nabla v(x) \cdot \nabla (v^{1-p}(x)w^p(x)) \d x \nonumber \\
    \leq & - (p-1)\int_{B_{8r}(x_0)} |\nabla \log v|^p w^p \d x+p\int_{B_{8r}(x_0)} v^{1-p}(x)|\nabla v(x)|^{p-1}|\nabla w(x)| w^{p-1}(x)\d x \nonumber \\
    \leq & - (p-1)\int_{B_{8r}(x_0)} |\nabla \log v|^p w^p \d x+p\int_{B_{8r}(x_0)} |\nabla \log v|^{p-1}w^{p-1}(x) \frac{1}{2^{p-1}} 2^{p-1}|\nabla w(x)| \d x \nonumber \\
    \leq & - (p-1)\int_{B_{8r}(x_0)} |\nabla \log v|^p w^p \d x+ \frac{p-1}{2} \int_{B_{8r}(x_0)} |\nabla \log v|^p w^p \d x\nonumber \\
    &+C(p) \int_{B_{8r}(x_0)} |\nabla w|^p \d x \nonumber \\
    \leq &-C \int_{B_{6r}(x_0)} |\nabla \log v|^p  \d x+ Cr^{N-p}.
\end{align}
By applying \eqref{eq6.5}, \eqref{eq6.6}, \eqref{eq6.7}, \eqref{eq6.8}, \eqref{eq6.9} in \eqref{eq6.4}, we obtain
\begin{align}\label{eq6.11}
 \int_{(0,1)} C_{N, s, p} & \bigg(\iint_{B_{6r}(x_0)\times B_{6r}(x_0)}\bigg|\log\bigg(\frac{v(x)}{v(y)}\bigg) \bigg|^p \frac{1}{|x-y|^{N+s p}} \d x \d y\bigg) \d \mu(s)\nonumber \\
   &+\alpha \int_{B_{6r}(x_0)} |\nabla \log v|^p  \d x \leq Cr^{N-p}.
\end{align}
For $\delta\in (0,\frac{1}{4})$ and $\zeta \geq 0$, we set 
$$w_1=:\bigg( \min \bigg\{ \log \frac{1}{2\delta}, \log\frac{\zeta+d_\epsilon}{v} \bigg\}\bigg)_+.$$ 
Observe that 
\begin{align}\label{eq6.120}
\{w_1=0\}=\{v\geq \zeta+d_\epsilon \}=\{u\geq \zeta \}.
\end{align}
From assumptions \eqref{eq6.1} and \eqref{eq6.120}, we get
 \begin{align}\label{eq6.15}
        |\{u\geq \zeta\} \cap B_{6r}(x_0)| \geq \frac{\sigma}{6^N} |B_{6r}(x_0)|.
    \end{align}
    Therefore, using \eqref{eq6.15}, we derive
\begin{align}\label{eq6.16}
    \log \frac{1}{2\delta}=&\frac{1}{|\{w_1=0\} \cap B_{6r}(x_0)|} \int_{\{w_1=0\} \cap B_{6r}(x_0)} \bigg( \log \frac{1}{2\delta} -w_1(x) \bigg) \d x \nonumber \\
    & \leq \frac{6^N}{\sigma |B_{6r}(x_0)|} \int_{B_{6r}(x_0)} \bigg( \log \frac{1}{2\delta} -w_1(x) \bigg) \d x \nonumber \\
    &\leq \frac{6^N}{\sigma } \bigg( \log \frac{1}{2\delta} -(w_1(x))_{B_{6r}(x_0)} \bigg), 
\end{align}
where $(w_1(x))_{B_{6r}(x_0)}=\frac{1}{|{B_{6r}(x_0)}|}\int_{B_{6r}(x_0)} w_1 \d x$.
From \eqref{eq6.11} and definition of $w_1$, we get
\begin{align}\label{eq6.12}
  { \int_{B_{6r}(x_0)} |\nabla w_1|^p  \d x\leq \int_{B_{6r}(x_0)} |\nabla \log v|^p  \d x \leq Cr^{N-p}.}
\end{align}
Using the Poinca\'re inequality \cite[Theorem 2]{EVANS2022} together with \eqref{eq6.12} and H\"older inequality, we deduce that
\begin{align}\label{eq6.14}
    \int_{B_{6r}(x_0)} |w_1 -(w_1)_{B_{6r}(x_0)}| \d x \leq C r^{1+\frac{N}{p'}}\bigg(\int_{B_{6r}(x_0)} |\nabla w_1|^p  \d x \bigg)^\frac{1}{p} \leq C r^N= C|B_{6r}(x_0)|,
\end{align}
where $p'=\frac{p}{p-1}$. Integrating both sides of \eqref{eq6.16} over $\{w_1=\log \frac{1}{2\delta}\} \cap \{B_{6r}(x_0)\}$ and using \eqref{eq6.14}, we get
\begin{align}\label{eq6.17}
    &\bigg|\bigg\{w_1=\log \frac{1}{2\delta}\bigg\} \cap \{B_{6r}(x_0)\} \bigg| \log \frac{1}{2\delta} \nonumber \\
    \leq &\frac{6^N}{\sigma } \int_{\{w_1=\log \frac{1}{2\delta}\} \cap \{B_{6r}(x_0)\}}    \bigg( \log \frac{1}{2\delta} -(w_1(x))_{B_{6r}(x_0)} \bigg) \d x \nonumber \\
    \leq & \frac{6^N}{\sigma } \int_{{B_{6r}(x_0)}}  | w_1 -(w_1(x))_{B_{6r}(x_0)} | \d x \nonumber \\ \leq & \frac{C}{\sigma}|B_{6r}(x_0)|.
\end{align}
Note that $\bigg\{w_1=\log \frac{1}{2\delta}\bigg\}=\{v\leq 2\delta (\zeta+d_\epsilon)\}=\{u\leq 2\delta\zeta+d_\epsilon(2\delta-1)\}$ and $\delta\in (0,\frac{1}{4})$. Therefore, using \eqref{eq6.17} together with the above fact, we conclude that \eqref{eq6.2} holds.
\end{proof}

\begin{lem}\label{lmn6.2}
    Let $u$ be a weak supersolution of \eqref{MP} such that $u\geq 0$ in $B_R(x_0)\subset \Omega$. Suppose that $\zeta\geq 0$ and there exists $\sigma\in (0,1]$ such that 
    \begin{align*}
        |\{u\geq \zeta\} \cap B_r(x_0)| \geq \sigma |B_r(x_0)|,
    \end{align*}
    for some $r\in( 0,1]$ with $0<16r<R$. Then there exists a constant $\delta \in (0,\frac{1}{4})$ such that
    \begin{align}\label{eq6.19}
        \essi_{B_{4r}(x_0)} u \geq \delta \zeta -\sup_{s\in \Sigma}\bigg(\frac{r}{R}\bigg)^{\frac{sp}{p-1}}  \operatorname{Tail}\big(u_-;x_0,R\big).
    \end{align}
\end{lem}
\begin{proof} First, we show that for every $\epsilon>0$, there exists a constant $\delta \in \bigg(0, \frac{1}{4}\bigg)$ such that
\begin{align}\label{eq6.20}
    \essi_{B_{4r}(x_0)} u \geq \delta \zeta -\sup_{s\in \Sigma}\bigg(\frac{r}{R}\bigg)^{\frac{sp}{p-1}}  \operatorname{Tail}\big(u_-;x_0,R\big)-2\epsilon.
\end{align}
Therefore, by passing to the limit in \eqref{eq6.20}, we obtain \eqref{eq6.19}. To prove \eqref{eq6.20}, let us suppose that 
  \begin{align}\label{eq6}
  \delta \zeta -\sup_{s\in \Sigma}\bigg(\frac{r}{R}\bigg)^{\frac{sp}{p-1}}  \operatorname{Tail}\big(u_-;x_0,R\big) -2\epsilon \geq0.
  \end{align}
 Otherwise, \eqref{eq6.19} holds trivially, since $u\geq 0$ in $B_R(x_0)$.
  \par For any $\gamma \in [r,6r]$, let $w\in C_c^\infty(B_\gamma(x_0))$ be a cutoff function satisfying $0\leq w\leq 1$. Let $l\in (\delta \zeta, 2 \delta \zeta )$, and choose the test function $\psi=vw^p$, where $v:=(l-u)+=(u-l)-$. Then we get
  \begin{align*}
      &0\leq \int_{(0,1)}C_{N, s, p}\left(\iint_{B_\gamma(x_0)\times B_\gamma(x_0)} {{ \mathcal{A}(u(x,y))(v(x)w^p(x)-v(y)w^p(y))}}\d \nu \right) \d \mu(s)\nonumber\\
& \quad+ 2\int_{(0,1)}C_{N, s, p}\left(\int_{\mathbb{R}^{N}\setminus B_\gamma(x_0)}\int_{B_\gamma(x_0)} { \mathcal{A}(u(x,y))v(x)w^p(x)}\d \nu \right) \d \mu(s)\nonumber\\
& \quad +\alpha \int_{B_\rho(x_0)} |\nabla u(x)|^{p-2}\nabla u(x) \cdot \nabla (v(x)(w^p(x)) \d x:=A_1+2A_2+\alpha A_3.
  \end{align*}
We will estimate only $A_2$ and the estimates for $A_1$ and $A_3$ can be obtained following similar arguments as in the proof of the Caccioppoli inequality in Lemma $3.3$.

\noindent \textbf{Estimate of $A_2$:} We have
            \begin{align}\label{eq6.22}
                A_2= &\int_{(0,1)}C_{N, s, p}\left(\int_{\mathbb{R}^N\setminus {B_{\gamma}(x_0)} \cap \{u(y)<0\}}\int_{B_{\gamma}(x_0)} {{ \mathcal{A}(u(x,y))(v(x)w^p(x))}}\d \nu\right) \d \mu(s)\nonumber\\
                &+ \int_{(0,1)}C_{N, s, p} \left(\int_{\mathbb{R}^N\setminus {B_{\gamma}(x_0)}\cap \{u(y)\geq 0\}}\int_{B_{\gamma}(x_0)} {{\mathcal{A}(u(x,y))(v(x)w^p(x))}}\d \nu\right) \d \mu(s)\nonumber\\
                =&A_2'+A_2''.
            \end{align}
\noindent \textbf{Estimate of $A_2'$:} Using the definition of $v$ and properties of $w$, we get
    \begin{align}\label{eq6.23}
        A_2'\leq &\int_{(0,1)}C_{N, s, p}\left(l\int_{\mathbb{R}^N\setminus {B_{\gamma}(x_0)} \cap \{u(y)<0\}}\int_{B_{\gamma}(x_0)\cap \{u<l\}} \frac{{(l+u_-(y))^{p-1}
        w^p(x)}{}}{|x-y|^{N+s p}} \d x \d y\right) \d \mu(s) \nonumber \\
        \leq & \int_{(0,1)} C_{N, s, p}l\left(\ess_{x \in \supp w}\int_{\mathbb{R}^N\setminus {B_{\gamma}(x_0)} }\frac{{(l+u_-(y))^{p-1}}{}}{|x-y|^{N+s p}} \d y \right) \nonumber \\
        &\times |B_{\gamma}(x_0)\cap \{u<l\}| \d \mu(s).
    \end{align}
    \noindent \textbf{Estimate of $A_2''$:} Similarly, using the definition of $v$ and properties of $w$, we obtain
    \begin{align} \label{eq6.24}
        A_2''\leq & \int_{(0,1)}C_{N, s, p}\left( l^{p-1}\int_{\mathbb{R}^N\setminus {B_{\gamma}(x_0)}\cap \{u(y)\geq 0\}}\int_{B_{\gamma}(x_0)\cap \{u<l\}} \frac{{}{(v(x)w^p(x))}}{|x-y|^{N+s p}} \d x \d y\right) \d \mu(s)\nonumber\\
        \leq & \int_{(0,1)}C_{N, s, p} l^{p} \left(\mathop{\mathrm{\ess}}_{x \in \supp w}\int_{\mathbb{R}^N\setminus {B_{\gamma}(x_0)}\}} \frac{{1}{}}{|x-y|^{N+s p}}  \d y\right) |B_{\gamma}(x_0)\cap \{u<l\}| \d \mu(s).
    \end{align}
    Therefore, using \eqref{eq6.22}, \eqref{eq6.23}, \eqref{eq6.24} and considering the proof of the Caccioppoli inequality in Lemma \ref{lmn3.2} for the estimating of $A_1$ and $A_3$, we deduce that
    \begin{align}\label{eq6.25}
        &\alpha \int_{B_\gamma(x_0)} w^p|\nabla v|^p \d x+\int_{(0,1)} C_{N,s,p} \left( \iint_{B_\gamma(x_0)\times B_\gamma(x_0)} \frac{|v(x) w(x)-v(y)w(y)|^p}{|x-y|^{N+sp}}\d x \d y\right) \d \mu(s)
     \nonumber\\
    \leq &  C \Bigg[ \int_{(0,1)} C_{N,s,p} \left( \iint_{B_\gamma(x_0)\times B_\gamma(x_0)} \frac{ \max\{v(x), v(y)\}^p|w(x)-w(y)|^p}{|x-y|^{N+sp}}\d x\d y\right)\d \mu(s) \nonumber \\
  & +\int_{(0,1)} C_{N, s, p} l\left(\ess_{x \in \supp w}\int_{\mathbb{R}^N\setminus {B_{\gamma}(x_0)} }\frac{{(l+u_-(y))^{p-1}}{}}{|x-y|^{N+s p}} \d y \right) |B_{\gamma}(x_0)\cap \{u<l\}| \d \mu(s) \nonumber\\
  &+\alpha \int_{B_\gamma(x_0)} v^p|\nabla w|^p \d x \Bigg].
    \end{align}
     For $j\in \mathbb{N}\cup \{0\}$, we set
    \begin{align*}
        l=\zeta_j=\delta \zeta+2^{-(j+1)} \delta \zeta,~ \quad \gamma=\gamma_j=4r+2^{1-j}r,~ \quad \bar{\gamma_j}=\frac{\gamma_j+\gamma_{j+1}}{2}.
    \end{align*}
    Then we have $l\in (\delta \zeta, 2\delta \zeta)$, {$\gamma_j \in (4r,6r]$}, $\bar{\gamma_j} \in (4r,6r)$ and 
    \begin{align*}
        \zeta_j-\zeta_{j+1}=2^{-(j+2)}\delta \zeta \geq 2^{-(j+3)}\zeta_j.
    \end{align*}
    We denote $B_j=B_{\gamma_j}(x_0)$, $\bar{B_i}=B_{\bar{\gamma_j}}(x_0)$. Also, note that 
    \begin{align*}
        v_j=(\zeta_j-u)_+\geq  2^{-(j+3)}\zeta_j \chi_{\{u<\zeta_{j+1}\}},
    \end{align*}
    where $\chi_A(x)$ is the characteristic function of a set. Furthermore, $\zeta_j$, $\gamma_j$ are monotonically decreasing and $\gamma_{j+1}<\bar{\gamma_j}<\gamma_j$. Let $w_j \in C_c^\infty(\bar{B_j})$ be a sequence of functions such that $w_i=1$ in $B_{j+1}$ and $0\leq w_i\leq 1$ with $|\nabla w_j|\leq \frac{2^{j+3}}{r}$ in $\bar{B_j}$. Choosing $l=\zeta_j$, $w=w_j$ and $v=v_j$ in $\eqref{eq6.25}$, we obtain
    \begin{align}\label{eq6.29}
         &\alpha \int_{B_{\gamma_j}(x_0)} w_j^p|\nabla v_j|^p \d x \leq  C \Bigg[ \alpha \int_{B_{\gamma_j}(x_0)} v_j^p|\nabla w_j|^p \d x \nonumber \\
     & + \int_{(0,1)} C_{N,s,p} \left( \iint_{B_{\gamma_j}(x_0)\times B_{\gamma_j}(x_0)} \frac{ \max\{v_j(x), v_j(y)\}^p|w_j(x)-w_j(y)|^p}{|x-y|^{N+sp}}\d x\d y\right)\d \mu(s) \nonumber \\
  & +\int_{(0,1)}C_{N, s, p}\zeta_j\left(\ess_{x \in \supp w_j}\int_{\mathbb{R}^N\setminus {B_{\gamma_j}(x_0)} }\frac{{ (\zeta_j+u_-(y))^{p-1}}{}}{|x-y|^{N+s p}} \d y \right) \d \mu(s)  \nonumber \\
  &\times|B_{\gamma_j}(x_0)\cap \{u<\zeta_j\}| \Bigg] :=\alpha B_1+B_2+ B_3.
    \end{align}
We now use Lemma \ref{MI Lemma} to prove our result.

\noindent \textbf{Estimate of $B_1$:} Using the properties of $w_j$ and definition of $v_j$, we get
 \begin{align}\label{eq6.30}
     B_1\leq \int_{B_{\gamma_j}(x_0)\cap \{u<\zeta_j\}} v_j^p 2^{(j+3)p}r^{-p} \d x \leq C 2^{jp} r^{-p} \zeta_j^p|B_{\gamma_j}(x_0)\cap \{u<\zeta_j\}|.
 \end{align}                       
\noindent \textbf{Estimate of $B_2$:} Similar arguments of \textbf{Estimate of $I_1$} in Lemma \ref{bddness} gives us
\begin{align}\label{eq6.31}
    B_2 \leq& \zeta_j^p 2^{(j+3)p} r^{-p}\int_{(0,1)} C_{N,s,p} \bigg( \iint_{B_{\gamma_j}(x_0)\times B_{\gamma_j}(x_0) \cap \{u<\zeta_j\}} |x-y|^{-N-sp+p} \d x \d y\bigg) \d \mu(s) \nonumber \\
    \leq & C \zeta_j^p 2^{jp} \sup_{s\in \Sigma}r^{-sp} |B_{\gamma_j}(x_0)\cap \{u<\zeta_j\}| \leq C  2^{jp} r^{-p} \zeta_j^p|B_{\gamma_j}(x_0)\cap \{u<\zeta_j\}|.
\end{align}
\noindent \textbf{Estimate of $B_3$:} Observe that for all $x\in \supp w_j=\bar{B_j}$ and for all $y\in \mathbb{R}^N\setminus B_j$, we have
\begin{align}\label{eq6.32}
    \frac{|y-x_0|}{|y-x|}\leq \frac{|y-x|+|x-x_0|}{|y-x|}\leq 1+\frac{\bar{\gamma_j}}{\gamma_j-\bar{\gamma_j}} \leq { 2^{j+5}}.
\end{align}
Thus using the convexity property and \eqref{eq6.32}, we get
    \begin{align}\label{eq6.33}
        B_3\leq & \zeta_j |B_{\gamma_j}(x_0)\cap \{u<\zeta_j\}| \int_{(0,1)}C_{N, s, p}\left(2^{(j+5)(N+sp)}\int_{\mathbb{R}^N\setminus {B_{\gamma_j}(x_0)} }\frac{{ (\zeta_j+u_-(y))^{p-1}}{}}{|y-x_0|^{N+s p}} \d y \right) \d \mu(s)\nonumber \\
       \leq & \zeta_j |B_{\gamma_j}(x_0)\cap \{u<\zeta_j\}| \int_{(0,1)}C_{N, s, p}\bigg(2^{(j+5)(N+sp)}  \bigg\{ \zeta_j^{p-1} {\frac{r^{-sp}}{sp} }\nonumber \\
       &+r^{-sp} \bigg(\frac{r}{R}\bigg)^{sp}R^{sp} \int_{\mathbb{R}^N\setminus {B_{R}(x_0)}} \frac{u_-(y))^{p-1}}{|y-x_0|^{N+s p}} \bigg\} \d \mu(s) \nonumber \\
       \leq &C\sup_{s\in \Sigma}2^{j(N+sp)} \zeta_j |B_{\gamma_j}(x_0)\cap \{u<\zeta_j\}| r^{-p}\bigg(\zeta_j^{p-1}+\sup_{s\in \Sigma} \bigg(\frac{r}{R}\bigg)^{sp}[\operatorname{Tail}(u_-;x_0,R)]^{p-1}\bigg)\nonumber \\
       \leq & C\sup_{s\in \Sigma}2^{j(N+sp)} \zeta_j^{p} |B_{\gamma_j}(x_0)\cap \{u<\zeta_j\}| r^{-p}.
    \end{align}
    By using \eqref{eq6.30}, \eqref{eq6.31} and \eqref{eq6.33} in \eqref{eq6.29}, we get 
    \begin{align}\label{eq6.34}
        \int_{B_{\gamma_j}(x_0)} w_j^p|\nabla v_j|^p \d x \leq  C \sup_{s\in \Sigma}2^{j(N+p+sp)} \zeta_j^{p} |B_{\gamma_j}(x_0)\cap \{u<\zeta_j\}| r^{-p}.
    \end{align}
    Again using convexity property, Sobolev inequality \eqref{sob}, \eqref{eq6.30} and \eqref{eq6.34}, we obtain
    \begin{align}\label{eq6.35}
         (\zeta_j-\zeta_{j+1})^p\left( \frac{|B_{{j+1}}(x_0)\cap \{u<\zeta_{j+1}\}|}{|B_{j+1}|}\right)^\frac{1}{\eta} =& \frac{1}{|B_{j+1}|^\frac{1}{\eta}} \left( \int_{B_{{j+1}}(x_0)\cap \{u<\zeta_{j+1}\}} (\zeta_j-\zeta_{j+1})^{p\eta} \d x \right)^\frac{1}{\eta}\nonumber\\
         \leq & \left( \fint_{B_{j+1}}v_j^{\eta p} \d x\right)^\frac{1}{\eta}=\left( \fint_{B_{j+1}}v_j^{\eta p} w_j^{\eta p}\d x\right)^\frac{1}{\eta} \nonumber \\
          \leq & C \left( \fint_{B_{j}}v_j^{\eta p} w_j^{\eta p}\d x\right)^\frac{1}{\eta} \nonumber \\
           \leq & C \gamma_j^p \fint_{B_{j}} |\nabla(v_j w_j)|^p \d x  \nonumber \\
          \leq & C r^p  \left( \fint_{{B_{j}}} {v_j}^p|\nabla w_j|^p \d x+ \fint_{{B^{j}}} w_j^p |\nabla{v_j}|^p \d x\right) \nonumber \\
          \leq &C\sup_{s\in \Sigma}2^{j(N+p+sp)} \zeta_j^{p} \frac{|B_{j}(x_0)\cap \{u<\zeta_j\}|}{|B_{j}(x_0)|}.
    \end{align}
   We set
    \begin{align}\label{eq6.36}
        P_j:=\frac{|B_{j}(x_0)\cap \{u<\zeta_j\}|}{|B_{j}(x_0)|}~~\text{for}~~j=0,1,2....
    \end{align}
   Thus from \eqref{eq6.35} and \eqref{eq6.36}, we have
    \begin{align*}
        P_{j+1} \leq C  2^{j\eta \sup_{s\in \Sigma}(N+2p+sp)}  P_{j}^\eta .
     \end{align*}
     Choose $C_1=C$, $C_2= 2^{\eta \sup_{s\in \Sigma}(N+2p+sp)}>1$ and $\beta+1=\eta$ in Lemma \ref{MI Lemma}. Therefore, from \eqref{eq6}, we get
     \begin{align}\label{eq6.39}
         \zeta_0= \frac{3}{2} \delta \zeta=2\delta \zeta -\frac{1}{2}\delta \zeta \leq 2\delta \zeta -\frac{1}{2}\sup_{s\in \Sigma}\bigg(\frac{r}{R}\bigg)^{\frac{sp}{p-1}}  \operatorname{Tail}\big(u_-;x_0,R\big) -\epsilon.
     \end{align}
     Using \eqref{eq6.39} and Lemma \ref{lmn6.1}, we get
     \begin{align}\label{eq6.40}
         P_0=\frac{\bigg|B_{6r}(x_0)\cap \{2\delta \zeta -\frac{1}{2}\sup_{s\in \Sigma}\bigg(\frac{r}{R}\bigg)^{\frac{sp}{p-1}}  \operatorname{Tail}\big(u_-;x_0,R\big) -\epsilon\}\bigg|}{|B_{6r}(x_0)|} \leq  \frac{C}{\sigma \log \frac{1}{2\delta}}
     \end{align}
     for all $\delta \in \bigg(0,\frac{1}{4}\bigg)$.
     We now choose 
     \begin{align}\label{eq6.41}
         0<\delta =\frac{1}{4} e^{\bigg(-\frac{CC_1^{\frac{1}{\beta}}C_2^{\frac{1}{\beta^2}}}{\sigma} \bigg)}<\frac{1}{4}.
     \end{align}
     Thus from \eqref{eq6.40} and \eqref{eq6.41}, we conclude that
     $$P_0 \leq C_1^{-\frac{1}{\beta}}C_2^{-\frac{1}{\beta^2}}.$$
     From Lemma \ref{MI Lemma}, we obtain $P_j \rightarrow 0$ as $j \rightarrow \infty$.
     Therefore, we get { $$\essi_{B_{4r}(x_0)} u \geq \delta \zeta.$$}
     This completes the proof.
\end{proof}
\subsection{Preliminary Version of Harnack Inequality}
In this subsection, we prove a preliminary version of the Harnack inequality using the Krylov-Safonov covering lemma \cite[Lemma $7.2$]{KS2001}. We first recall the Krylov-Safonov covering lemma from \cite[Lemma $2.5$]{DKP2014}.

\begin{lem}[Lemma $2.5$,\cite{DKP2014}]\label{lmn6.30}
    Let $A\subset B_r(x_0)$ be a measurable set with $r\in (0,1]$ and $\delta'\in(0,1)$. Set
    \begin{align*}
        A_{\delta'}:=\bigcup_{\gamma>0}\bigg\{B_{3\gamma}(x)\cap B_r(x_0): x\in B_r(x_0) ~\&~ |A\cap B_{3\gamma}(x)| >{\delta'}|B_{\gamma}(x)|\bigg\}.
    \end{align*}
    Then either $|A_{\delta'}| \geq \frac{C}{\delta'}|A|$ or $A_{\delta'}=B_r(x_0)$,
    where $C=C(N)$ is a constant such that $C\in (0,1]$.
\end{lem}

The following lemma is about the preliminary version of Harnack inequality, which is proved using arguments similar to those in \cite[Lemma 4.1]{DKP2014}. For completeness, we provide a brief sketch of the proof.
\begin{lem}\label{lmn6.4}
Let $u$ be a weak supersolution of \eqref{MP} with $u\geq 0$ in $B_R(x_0)\subset \Omega$. Then there exist positive constants $\kappa:=\kappa(N,p,\mu,\Sigma) \in(0,1)$ and $C:=C(N,p,\mu,\Sigma)\in [1, \infty)$ such that
    \begin{align*}
         \bigg(\fint_{B_{{r}}(x_0)}u^\kappa \d x \bigg)^{\frac{1}{\kappa}} \leq C \essi_{B_r(x_0)}u +C{\sup_{s\in \Sigma}\bigg(\frac{r}{R}\bigg)^{\frac{sp}{p-1}}} \operatorname{Tail}\big(u_-;x_0,R\big),
    \end{align*}
    where $B_r(x_0) \subset B_R(x_0)$ with $r\in (0,1]$.
\end{lem}

\begin{proof}
 For $\epsilon>0$ and $i\in \mathcal{N}\cup \{0\}$, we define
 \begin{align*}
     E^i_t:=\bigg\{x\in B_r(x_0):u(x)>\epsilon \delta^i -\frac{F}{1-\delta} \bigg\},
 \end{align*}
 where $\delta$ is given in Lemma \ref{lmn6.2} and
 \begin{align*}
     F:=\sup_{s\in \Sigma}\bigg(\frac{r}{R}\bigg)^{\frac{sp}{p-1}}\operatorname{Tail}\big(u_-;x_0,R\big).
 \end{align*}
 The remainder of the proof follows the same argument as in \cite[Lemma 4.1]{DKP2014}, together with Lemma \ref{lmn6.30}. This completes the proof.
\end{proof}
We next prove the following Caccioppoli-type estimate.
\begin{lem}\label{lmn7.1}
    Let $p\in(1, \infty)$, $d>0$ and $q\in(1,p)$. Let $u$ be a weak supersolution of \eqref{MP} such that $u\geq 0$ in $B_R(x_0)\subset \Omega$ and $\phi=(u+d)^\frac{p-q}{p}$. Then there exists a constant {$C:=C(p,q,N, \Sigma,\mu)>0$} such that
    \begin{align*}
       \alpha & \int_{B_r(x_0)} w^p|\nabla \phi|^p \d x \leq C\bigg[\alpha \int_{B_r(x_0)} \phi^p|\nabla w|^p \d x \nonumber \\
       &+ \int_{(0,1)} C_{N,s,p} \left( \iint_{B_r(x_0)\times B_r(x_0)} \frac{ \max\{\phi(x), \phi(y)\}^p|w(x)-w(y)|^p}{|x-y|^{N+sp}}\d x\d y\right)\d \mu \bigg] \nonumber \\
       &+C \bigg[ \int_{(0,1)} C_{N,s,p} \left( \ess_{x\in \supp w} \int_{\mathbb{R}^N\setminus B_r(x_0)}\frac{1}{|x-y|^{N+sp}} \d y \right) \d \mu \nonumber \\
       & + d^{1-p} \sup_{s\in \Sigma}\bigg(\frac{r}{R}\bigg)^{sp}\sup_{s\in \Sigma}r^{-sp}[\operatorname{Tail}\big(u_-(y);x_0,R\big)]^{p-1} \bigg]\bigg(\int_{B_r(x_0)}\phi^p w^p \d x \bigg)
    \end{align*}
    for any $B_r(x_0)\subset B_{\frac{3R}{4}}(x_0)$ and nonnegative function $w\in C_c^\infty(B_r(x_0))$.
\end{lem}
\begin{proof}
    Let $v=u+d$ and $q\in [1+\epsilon,p-\epsilon]$ for $d>0$ and sufficiently small $\epsilon>0$. Then $v$ is a weak supersolution of \eqref{MP}. By choosing $\psi=v^{1-q}w^p$ as a test function, we get
    \begin{align}
        0 \leq& \int_{(0,1)}C_{N, s, p}\left(\iint_{\mathbb{R}^{2 N}} {{ \mathcal{A}(v(x,y))(v(x)^{1-q}w^p(x)-v(y)^{1-q}w^p(y))}}\d \nu \right) \d \mu(s)\nonumber\\
       & +\alpha \int_{\Omega} |\nabla v(x)|^{p-2}\nabla v(x) \cdot \nabla (v^{1-q}(x)w^p(x)) \d x \nonumber \\
        =& \int_{(0,1)}C_{N, s, p} \left(\iint_{B_r(x_0)\times B_r(x_0)} {{\mathcal{A}(v(x,y))(v^{1-q}(x)w^p(x)-v^{1-q}(y)w^p(y))}}\d \nu\right) \d \mu(s)\nonumber\\
        &+2 \int_{(0,1)}C_{N, s, p}\left(\int_{\mathbb{R}^N\setminus {B_r(x_0)}}\int_{B_r(x_0)} {{\mathcal{A}(v(x,y))(v^{1-q}(x)w^p(x))}}\d \nu\right) \d \mu(s)\nonumber\\
        & + \alpha \int_{B_r(x_0)} |\nabla v(x)|^{p-2}\nabla v(x) \cdot \nabla (v^{1-q}(x)w^p(x)) \d x:= J_1+2J_2+\alpha J_3. \label{eq7.3}
    \end{align}
   \noindent \textbf{Estimate of $J_1$:} Following similar arguments from \cite[Lemma $5.1$]{DKP2014}, combined with Lemma \ref{lmn2.9}, there exists a constant $C:=C(p,q)$ such that
    \begin{align*}
       & \iint_{B_r(x_0)\times B_r(x_0)} {{\mathcal{A}(v(x,y))(v^{1-q}(x)w^p(x)-v^{1-q}(y)w^p(y))}}\d \nu\nonumber \\
        \leq& -C\iint_{B_r(x_0) \times B_r(x_0)} {|\phi(x)-\phi(y)|^p w(y)^p }\d \nu\nonumber \\
        &+ C \iint_{B_r(x_0)\times B_r(x_0)} { \max\{\phi(x), \phi(y)\}^p|w(x)-w(y)|^p}\d \nu.
    \end{align*}
    Therefore, we have
    \begin{align}
        J_1 \leq& -C \int_{(0,1)}C_{N, s, p} \bigg(\iint_{B_r(x_0) \times B_r(x_0)} {|\phi(x)-\phi(y)|^p w(y)^p }\d \nu \bigg) \d \mu\nonumber \\
        &+ C\int_{(0,1)} C_{N,s,p} \left( \iint_{B_r(x_0)\times B_r(x_0)} { \max\{\phi(x), \phi(y)\}^p|w(x)-w(y)|^p}\d \nu\right)\d \mu. \label{eq7.4}
    \end{align}
    \noindent \textbf{Estimate of $J_2$:} Observe that \begin{enumerate}
        \item[(i)] $|v(x)-v(y)|^{p-2}(v(x)-v(y))\leq C(v(x))^{p-1}+C(v_-(y))^{p-1}$,
        \item[(ii)] $v^{1-q}(x)\leq d^{1-p}v^{p-q}(x)$,
        \item[(iii)] $(v_-(y))=0$ for all $y\in B_R(x_0)$.
     \end{enumerate}
     Using the above facts, we derive 
    \begin{align}
        J_2 \leq& \int_{(0,1)}C_{N, s, p}\left(\int_{\mathbb{R}^N\setminus {B_r(x_0)}}\int_{B_r(x_0)}{C\{(v(x))^{p-1} + v^{p-1}_-(y)\} v^{1-q}(x)w^p(x)}\d \nu \right) \d \mu \nonumber \\
        \leq & C \int_{(0,1)}C_{N, s, p} \bigg( \ess_{x\in \supp w} \int_{\mathbb{R}^N\setminus B_r(x_0)} \frac{1}{|x-y|^{N+sp}}\d y\nonumber \\
        &+d^{1-p}{4^{N+sp}} \int_{\mathbb{R}^N\setminus B_R(x_0)} (v_-(y))^{p-1} {\frac{1}{|y-x_0|^{N+sp}}} \d y \bigg)\bigg(\int_{B_r(x_0)}\phi^p w^p \d x \bigg) \d \mu \nonumber \\
        \leq & C \int_{(0,1)}C_{N, s, p} \bigg( \ess_{x\in \supp w} \int_{\mathbb{R}^N\setminus B_r(x_0)} \frac{1}{|x-y|^{N+sp}}\d y\bigg) \bigg(\int_{B_r(x_0)}\phi^p w^p \d x \bigg) \d \mu \nonumber \\
        &+C d^{1-p} {\sup_{s\in \Sigma} 4^{N+sp} } \sup_{s\in \Sigma}\bigg(\frac{r}{R}\bigg)^{sp}\sup_{s\in \Sigma}r^{-sp} [\operatorname{Tail}\big(v_-(y);x_0,R\big)]^{p-1}  \bigg(\int_{B_r(x_0)}\phi^p w^p \d x \bigg)\nonumber \\
        \leq &  C \bigg\{\int_{(0,1)}C_{N, s, p} \bigg( \ess_{x\in \supp w} \int_{\mathbb{R}^N\setminus B_r(x_0)} \frac{1}{|x-y|^{N+sp}}\d y \bigg) \d \mu \nonumber \\
        &+ d^{1-p} {\sup_{s\in \Sigma}\bigg(\frac{r}{R}\bigg)^{sp}}\sup_{s\in \Sigma}r^{-sp}[\operatorname{Tail}\big(v_-(y);x_0,R\big)]^{p-1} \bigg\} \bigg(\int_{B_r(x_0)}\phi^p w^p \d x \bigg). \label{eq7.5}
    \end{align}
    \noindent \textbf{Estimate of $J_3$:} Using Young's inequality, we get
    \begin{align}
    J_3\leq& (1-q) \int_{B_r(x_0)} |\nabla v|^{p} w^p v^{-q}\d x+p\int_{B_r(x_0)}|\nabla v|^{p-1} |\nabla w| v^{1-q} w^{p-1} \d x \nonumber \\
    =& -(q-1) \int_{B_r(x_0)} |\nabla v|^{p} w^p v^{-q}\d x\nonumber \\
    &+p\int_{B_r(x_0)} |\nabla w| v^{\frac{p-q}{p}}\left(\frac{q-1}{2(p-1)}\right)^{-\frac{p-1}{p}} \left(\frac{q-1}{2(p-1)}\right)^{\frac{p-1}{p}}   v^{\frac{q(1-p)}{p}} |\nabla v|^{p-1} w^{p-1} \d x \nonumber \\
    \leq& -(q-1) \int_{B_r(x_0)} |\nabla v|^{p} w^p v^{-q}\d x + C\int_{B_r(x_0)} |\nabla w|^p v^{{p-q}} \d x\nonumber \\
    &+ \frac{q-1}{2}  \int_{B_r(x_0)} |\nabla v|^{p} w^p v^{-q}\d x \nonumber \\
    =& -\frac{q-1}{2} \bigg(\frac{p}{p-q} \bigg)^p \int_{B_r(x_0)} |\nabla \phi|^{p} w^p \d x+C\int_{B_r(x_0)} |\nabla w|^p \phi^{{p}} \d x. \label{eq7.6}
    \end{align}
   Therefore, substituting \eqref{eq7.4}, \eqref{eq7.5}, and \eqref{eq7.6} into \eqref{eq7.3}, we obtain the desired conclusion. This completes the proof.
\end{proof}
We now prove the following weak Harnack inequality for supersolutions of \eqref{MP}.
\begin{thm}
    Let $u$ be a weak supersolution of \eqref{MP} with $u\geq 0$ in $B_R(x_0) \subset \Omega$. Then there exists a constant $C:=C(N,p,s,\Sigma)>0$ such that
    \begin{align*}
        \bigg(\fint_{B_{\frac{r}{2}}(x_0)}u^t \d x \bigg)^\frac{1}{t} \leq C \essi_{B_r(x_0)}u +C{\sup_{s\in \Sigma}\bigg(\frac{r}{R}\bigg)^{\frac{sp}{p-1}}} \operatorname{Tail}\big(u_-(y);x_0,R\big),
    \end{align*}
    where $r\in(0,1]$, $B_r(x_0)\subset B_{\frac{R}{2}}(x_0)$ and $0<t<\eta(p-1)$.
    \begin{proof}
        We prove the result for the case $1<p<N$. The case $p\geq N$ can be treated similarly, yielding the same conclusion. Let $0<r\leq 1$ and $\frac{1}{2}<m'<m\leq \frac{3}{4}$. Let $w\in C_c^\infty(B_{mr}(x_0))$ be such that $w(x)=1$ for all $x\in B_{m'r}(x_0)$, $0\leq w(x)\leq 1$ for all $x\in B_{mr}(x_0)$ with $|\nabla w|\leq \frac{4}{(m-m')r}$. We set $v=u+d$ and $\phi=v^{\frac{p-q}{p}}$ for $q\in (1,p)$ and $d>0$. By using the property of $w$, we have
        \begin{align}\label{eq7.8}
            F_1:=\int_{B_r(x_0)} \phi^p|\nabla w|^p \d x \leq \frac{Cr^{-p}}{(m-m')^p} \int_{B_{mr}(x_0)} \phi^p \d x.
        \end{align}
        Following a similar arguments as in \textbf{Estimate $I_1$} of Theorem \ref{bddness}, we get
        \begin{align}\label{eq7.9}
            F_2:=&\int_{(0,1)} C_{N,s,p} \left( \iint_{B_r(x_0)\times B_r(x_0)} \frac{ \max\{\phi(x), \phi(y)\}^p|w(x)-w(y)|^p}{|x-y|^{N+sp}}\d x\d y\right)\d \mu \nonumber \\
            \leq& \frac{Cr^{-p}}{(m-m')^p} \sup_{s\in \Sigma}r^{p-sp}\bigg( \int_{B_{mr}(x_0)} \phi^p \d x\bigg)\mu\{(0,1)\} \nonumber \\
           \leq& \frac{Cr^{-p}}{(m-m')^p} \bigg( \int_{B_{mr}(x_0)} \phi^p \d x\bigg).
        \end{align}
        Now observe that 
        \begin{align}\label{eq7.10}
             & {\int_{(0,1)} C_{N,s,p} \left( \ess_{x\in \supp w} \int_{\mathbb{R}^N\setminus B_r(x_0)}\frac{1}{|x-y|^{N+sp}} \d y\right) \d \mu }\nonumber \\
            \leq & C \sup_{s\in \Sigma}r^{-sp} \mu\{(0,1)\} \leq C r^{-p}.
        \end{align}
       Suppose that $\operatorname{Tail}\big(u_-(y);x_0,R\big)$ is positive. For sufficiently small $\epsilon>0$, we define
       \begin{align}\label{eq7.11}
           d:=\frac{1}{2}\sup_{s\in \Sigma}\bigg(\frac{r}{R}\bigg)^{\frac{sp}{p-1}}  \operatorname{Tail}\big(u_-(y);x_0,R\big)+\epsilon >0.
        \end{align}
        Thus, using \eqref{eq7.10} and \eqref{eq7.11}, we obtain
        \begin{align}\label{eq7.12}
            F_3:=&\bigg[ \int_{(0,1)} C_{N,s,p} \left( \ess_{x\in \supp w} \int_{\mathbb{R}^N\setminus B_r(x_0)}\frac{1}{|x-y|^{N+sp}} \right) \d \mu \nonumber \\
       & + d^{1-p}\sup_{s\in \Sigma}\bigg(\frac{r}{R}\bigg)^{sp}\sup_{s\in \Sigma}r^{-sp}[\operatorname{Tail}\big(u_-(y);x_0,R\big)]^{p-1} \bigg]\bigg(\int_{B_r(x_0)}\phi^p w^p \d x \bigg) \nonumber \\
       \leq & \frac{Cr^{-p}}{(m-m')^p} \bigg( \int_{B_{mr}(x_0)} \phi^p \d x\bigg).
        \end{align}
        Also, note that if $\operatorname{Tail}\big(u_-(y);x_0,R\big)$ is zero, then $d=\epsilon>0$. Again, using \eqref{eq7.10}, we obtain the same estimate as \eqref{eq7.12}. Now, applying the Sobolev inequality, along with \eqref{eq7.8}, \eqref{eq7.9}, \eqref{eq7.12} and Lemma \ref{lmn7.1}, we deduce
        \begin{align}\label{eq7.13}
            \bigg(\fint_{B_{m'r}(x_0)} v^{\frac{N(p-q)}{N-p}} \d x\bigg)^\frac{p}{p^*}= \bigg(\fint_{B_{m'r}(x_0)} \phi^{p^*} \d x\bigg)^\frac{p}{p^*} \leq &\bigg(\fint_{B_{mr}(x_0)} |\phi w|^{p^*} \d x\bigg)^\frac{p}{p^*} \nonumber \\
            \leq &(mr)^{p-N} \int_{B_{mr}(x_0)} |\nabla(\phi w)|^{p} \d x \nonumber \\
            \leq & \frac{C}{(m-m')^p} \bigg( \fint_{B_{mr}(x_0)} \phi^p \d x\bigg),
        \end{align} 
        where $p^*=\frac{pN}{N-p}$. Since, $q\in (1,p)$, using \eqref{eq7.13} and Moser iteration technique from \cite[Theorem $1.2$]{GT2001}, we get
        \begin{align}\label{eq7.14}
            \bigg(\fint_{B_{\frac{r}{2}}(x_0)}v^t \d x \bigg)^{\frac{1}{t}} \leq C \bigg(\fint_{B_{\frac{3r}{4}}(x_0)}v^{t'} \d x \bigg)^{\frac{1}{t'}}, ~0<t'<t<\frac{N(p-1)}{N-p}. 
        \end{align}
        Note that {$u_-\geq v_-$} and
        \begin{align}\label{eq7.15}
            \bigg(\fint_{B_{\frac{r}{2}}(x_0)}u^t \d x \bigg)^{\frac{1}{t}} \leq \bigg(\fint_{B_{\frac{r}{2}}(x_0)}v^t \d x \bigg)^{\frac{1}{t}}.
        \end{align}
        Let {$\kappa \in(0,1) $ in Lemma \ref{lmn6.4} and $t'=\kappa$}. Using \eqref{eq7.14} and \eqref{eq7.15}, we obtain
        \begin{align}\label{eq7.16}
            \bigg(\fint_{B_{\frac{r}{2}}(x_0)}u^t \d x \bigg)^{\frac{1}{t}} \leq & C \essi_{B_r(x_0)}v +C{\sup_{s\in \Sigma}\bigg(\frac{r}{R}\bigg)^{\frac{sp}{p-1}}} \operatorname{Tail}\big(v_-(y);x_0,R\big)\nonumber \\
            \leq & C \essi_{B_r(x_0)}v +C{\sup_{s\in \Sigma}\bigg(\frac{r}{R}\bigg)^{\frac{sp}{p-1}}} \operatorname{Tail}\big(u_-(y);x_0,R\big)
        \end{align}
        for any $0<t<\frac{N(p-1)}{N-p}.$ 
       \noindent Observe that for sufficiently small $\epsilon$, in \eqref{eq7.16}, $$ {d=\frac{1}{2}\sup_{s\in \Sigma}\bigg(\frac{r}{R}\bigg)^{\frac{sp}{p-1}}  \operatorname{Tail}\big(u_-(y);x_0,R\big)+\epsilon}.$$ Therefore, as $\epsilon \rightarrow 0$, we obtain the desired conclusion. This completes the proof.
    \end{proof}
\end{thm}

\section{Lower and Upper Semicontinuity} \label{sec7}
In this section, we examine the pointwise behavior of the solution. Let $u$ be a measurable function that is locally essentially bounded from below in $\Omega$, $B_\gamma(y)\subset \Omega$ with $\gamma \in (0,1]$, $a,b\in(0,1)$, $K>0$, and $\beta_- \leq \essi_{B_\gamma(y)}u$.


Inspired by Liao \cite{L2021}, we say that $u$ satisfies property $\mathcal{(D)}$: ``\textit{if there exists a constant $\rho \in(0,1)$, (depend on $a,K, \beta_-,$ and other data but independent of $\gamma$; sec also Lemma \ref{lmn8.3}) such that 
\begin{align*} 
|\{u\leq \beta_-+K\}\cap B_\gamma(y)| \leq \rho |B_\gamma(y)|, 
\end{align*} 
then it implies that $u\geq \beta_-+aK$ almost everywhere in $B_{b\gamma}(y).$}" 
\par The lower semicontinuous regularization of $u$ is defined by
\begin{align*}
    u_*(x):=\esslim_{y\rightarrow x}u(y)=\lim_{r\rightarrow 0} \essi_{y\in B_r(x)} u(y), \text{ for }x\in \Omega.
\end{align*}
\begin{rem}
   As $u$ is locally essentially bounded from below in $\Omega$, $u_*$ is well-defined at every point in $\Omega$. In addition, $u_*$ is lower semicontinuous.
\end{rem}
Let $u\in L^1_{loc}(\Omega)$. We denote by $\mathcal{F}$, the set of Lebesgue points of $u$, namely,
\begin{align*}
    \mathcal{F}=\bigg \{ x\in \Omega: |u(x)|<\infty, \lim_{r\rightarrow 0} \fint_{B_r(x)} |u(x)-u(y)| \d y=0 \bigg\}.
\end{align*}
Observe that, by Lebesgue differentiation theorem, we have $|\mathcal{F}|=|\Omega|$.
\begin{lem}[Theorem $2.1$, \cite{L2021}]\label{lmn8.2}
    Let $u$ be a measurable function, which satisfies the property $\mathcal{(D)}$, locally integrable and locally essentially bounded from below in $\Omega$. Then 
    \begin{align*}
        u(x)=u_*(x)
        \text{ for every } x\in \mathcal{F}.
        \end{align*}
        In particular, $u_*$ is a lower semicontinuous representation of $u$ in $\Omega$.
\end{lem}
Next, we prove a De Giorgi type lemma for weak supersolution of \eqref{MP}.
\begin{lem}\label{lmn8.3}
    Let $u$ be a weak supersolution of \eqref{MP}. Let $B_r(x_0)\subset \Omega$ with $r \in (0,1]$, $a\in(0,1)$, $K>0$, $\beta_- \leq \essi_{B_r(x_0)}u$ and $\lambda_-\leq \essi_{\mathbb{R}^N}u$. Then there exists a constant $\rho=\rho(N,p,\Sigma, a, K, \beta_-, \lambda_-)\in (0,1)$ such that
    \begin{align}\label{eq8.5}
        |\{u\leq \beta_-+K\} \cap B_r(x_0)| \leq \rho|B_r(x_0)|,
    \end{align}
    implies $u\geq \beta_-+aK$ almost everywhere in $B_{\frac{3r}{4}}(x_0).$
\end{lem}
\begin{proof}
    Without loss of generality, let $x_0=0$.
    For $j\in \mathbb{N}\cup\{0\} $ and $r\in (0,1]$, we set
    $$r_j={r}\left(\frac{3}{4}+\frac{1}{2^{j+2}}\right),~\bar{r_j}=\frac{r_j+r_{j+1}}{2},~ B_j=B_{r_j}(0), \bar{B_j}=B_{\bar{r_j}}(0),$$
    and 
    $$\gamma_j=\beta_-+aK+\frac{(1-a)M}{2^j},~\bar{\gamma_j}=\frac{\gamma_j+\gamma_{j+1}}{2},~\phi_j=(\gamma_j-u)_+,~\bar{\phi_j}=(\bar{\gamma_j}-u)_+.$$
    
    We observe that $\bar{\gamma}_{j}<\gamma_j$, $B_{j+1}\subset \bar{B}_j\subset B_j$ and $\bar{\phi}_j\leq \phi_j$. Let us consider a function $w_j\in C_c^\infty(\bar{B_j})$ such that {$0\leq w_j\leq 1$} in ${B_j}$,
     $w_j=1$ in $B_{j+1}$ and $|\nabla w_j|\leq \frac{2^{j+3}}{r}$. 
     Given that, $u\geq \lambda_-$ in $\mathbb{R}^N$, we apply the definitions of $\gamma_j$ and $\phi_j$, to obtain
     \begin{align}\label{eq8.6}
         \phi_j=(\gamma_j-u)_+ \leq (K+\beta_--\lambda_-)_+=L \text{ in } \mathbb{R}^N.
     \end{align}
     Let $$X_j=\{u<\gamma_j\}\cap B_j.$$
    Using the Caccioppoli estimate from Lemma \ref{lmn3.2}, we derive
     \begin{align}\label{eq8.7}
        \alpha & \int_{B_j} w_j^p|\nabla \phi_j|^p \d x \leq  C \Bigg[  \alpha \int_{B_j} \phi_j^p|\nabla w_j|^p \d x \nonumber\\
    & +\int_{(0,1)} C_{N,s,p} \left( \iint_{B_j\times B_j} { \max\{\phi_j(x), \phi_j(y)\}^p|w_j(x)-w_j(y)|^p}\d \nu\right)\d \mu \nonumber \\
  & + \int_{(0,1)} C_{N,s,p} \left( \ess_{x\in \supp w_j} \int_{\mathbb{R}^N\setminus B_j} \frac{\phi_j^{p-1}(y)}{|x-y|^{N+sp}}\d y \cdot \int_{B_j}\phi_j w_j^p \d x \right)\d \mu \Bigg] \nonumber \\
  &=C(\alpha H_1+H_2+H_3).
     \end{align}
     \noindent  \textbf{Estimate of $H_1$:} Using \eqref{eq8.6} and properties of $w_j$, we obtain
\begin{align}\label{eq8.8}
    H_1=\int_{B_j} \phi_j^p|\nabla w_j|^p \d x \leq C2^{jp}L^p r^{-p}|X_j|.
\end{align}
    \noindent  \textbf{Estimate of $H_2$:}  From \eqref{eq8.6} and the properties of $w_j$, we get 
\begin{align}\label{eq8.9}
    H_2=&\int_{(0,1)} C_{N,s,p} \left( \iint_{B_j\times B_j} \frac{ \max\{\phi_j(x), \phi_j(y)\}^p|w_j(x)-w_j(y)|^p}{|x-y|^{N+sp}}\d x\d y\right)\d \mu \nonumber \\
    \leq & {CL^p \int_{(0,1)} C_{N,s,p} \bigg( \int_{X_j}\int_{B_j}\frac{2^{(j+3)p}}{r^p} |x-y|^{-N+(p-sp)} \d x \d y\bigg) \d \mu} \nonumber \\
    \leq & C2^{jp}L^p \sup_{s\in \Sigma}r^{-sp} |X_j|\mu\{(0,1)\}     \leq C2^{jp}L^p r^{-p}|X_j|.
\end{align}
 \noindent  \textbf{Estimate of $H_3$:} For all $x\in \supp w_j \text{ and }  y\in \mathbb{R}^N\setminus B_j$, we have
\begin{align}\label{eq8.10}
    \frac{1}{|x-y|}\leq \frac{1}{|y|}\bigg( \frac{|x|}{|x-y|}+1\bigg)\leq \frac{1}{|y|} (2^{j+5}+1)\leq \frac{2^{j+6}}{|y|}.
\end{align}
By utilizing \eqref{eq8.6}, \eqref{eq8.10}, and the properties of $w_j$, we deduce that
\begin{align}\label{eq8.11}
    H_3=&\int_{(0,1)} C_{N,s,p} \left( \ess_{x\in \supp w_j} \int_{\mathbb{R}^N\setminus B_j} \frac{\phi_j^{p-1}(y)}{|x-y|^{N+sp}}\d y \cdot \int_{B_j}\phi_j w_j^p \d x \right)\d \mu \nonumber \\
    \leq & C L^p \sup_{s\in \Sigma} 2^{j(N+sp)}  \sup_{s\in \Sigma}r^{-sp} |X_j|\mu\{(0,1)\} \leq C L^p \sup_{s\in \Sigma} 2^{j(N+sp)} r^{-p} |X_j|.
\end{align}
Using \eqref{eq8.8}, \eqref{eq8.9}, and \eqref{eq8.11} in \eqref{eq8.7}, we derive
\begin{align}\label{eq8.12}
    \int_{B_j} w_j^p|\nabla \phi_j|^p \d x \leq C L^p \sup_{s\in \Sigma} 2^{j(N+p+sp)} r^{-p} |X_j|=C L^p r^{-p} 2^{j(N+p+\sup_{s\in \Sigma}sp)}|X_j|.
\end{align}
Employing \eqref{eq8.8}, \eqref{eq8.12}, the Sobolev inequality \eqref{sob}, and Hölder's inequality, we get
\begin{align}\label{eq8.13}
    \frac{(1-a)K}{2^{j+2}}|X_{j+1}|=&\int_{X_{j+1}}(\bar{\gamma}_{j}-\gamma_{j+1}) \d x \leq \int_{B_{j+1}}\bar{\phi}_j \d x \leq \int_{B_{j+1}}{\phi}_j \d x  \nonumber \\
    \leq &|X_j|^{1-\frac{1}{p\eta}} \bigg( \int_{B_{j}}\phi_j^{p\eta} w_j^{p\eta} \d x \bigg)^{\frac{1}{p\eta}} \nonumber \\
    \leq &C r^{1-\frac{N}{p}+\frac{N}{p\eta}}|X_j|^{1-\frac{1}{p\eta}}\bigg( \int_{B_{j}}|\nabla(\phi_j w_j)|^p \d x \bigg)^{\frac{1}{p}} \nonumber \\
    \leq& CL r^{-\frac{N}{p}+\frac{N}{p\eta}} 2^{\frac{j(N+p+\sup_{s\in \Sigma}sp)}{p}} |X_j|^{1-\frac{1}{p\eta}+\frac{1}{p}}.
\end{align}
Dividing both sides of \eqref{eq8.13} by $|B_{j+1}|$, we obtain
\begin{align}\label{eq8.14}
    \frac{|X_{j+1}|}{|B_{j+1}|} \leq C\frac{L r^{-\frac{N}{p}+\frac{N}{p\eta}}}{(1-a)K} 2^{j\big(2+\frac{N+\sup_{s\in \Sigma}sp}{p}\big)} \frac{|X_j|^{1-\frac{1}{p\eta}+\frac{1}{p}}}{|B_{j+1}|}.
\end{align}
We set $P_j=\frac{|X_{j}|}{|B_{j}|}$. Using the facts $|B_j|<2^N|B_{j+1}|$ and $r_j<r$ in \eqref{eq8.14}, we deduce
\begin{align*}
    Y_{j+1} \leq  C\frac{L }{(1-a)K} 2^{j\big(2+\frac{N+\sup_{s\in \Sigma}sp}{p}\big)} Y_j^{^{1-\frac{1}{p\eta}+\frac{1}{p}}}.
\end{align*}
By setting
$$C_1= C\frac{L }{(1-a)K}, C_1=2^{\big(2+\frac{N+\sup_{s\in \Sigma}sp}{p}\big)}, \beta=\frac{1}{p}-\frac{1}{p\eta}, \rho=C_1^{-\frac{1}{\beta}}C_2^{-\frac{1}{\beta^2}},$$ 
and applying the condition \eqref{eq8.5}, we derive 
$$P_0\leq \rho.$$
Therefore, by Lemma \ref{MI Lemma}, we deduce that $P_{j}\rightarrow 0$ as $j\rightarrow\infty$. Consequently, we have $u\geq \beta_-+aK$ almost everywhere in $B_{\frac{3r}{4}}(0)$. This completes the proof.
\end{proof}
We finally prove the following two main results concerning lower semicontinuity and upper semicontinuity as consequences of Lemmas \ref{lmn8.2} and \ref{lmn8.3}.
\begin{thm}\label{thm8.4}
    Let $u$ be a weak supersolution of \eqref{MP}. Then
    \begin{align*}
        u(x)=u_*(x)  \text{ for every } x\in \mathcal{F}.
    \end{align*}
     In particular, $u_*$ is a lower semicontinuous representation of $u$ in $\Omega$.
\end{thm}
As an immediate consequence of Theorem \ref{thm8.4}, we obtain the following result.
\begin{cor} \label{uppsem}
     Let $u$ be a weak subsolution of \eqref{MP}. Then
    \begin{align*}
        u(x)=u^*(x)  \text{ for every } x\in \mathcal{F}.
    \end{align*}
     In particular, $u^*$ is an upper semicontinuous representation of $u$ in $\Omega$.
\end{cor}

\noindent \textbf{Conflict of interest statement:} On behalf of the authors, the corresponding author states that there is no conflict of interest.
\newline
\textbf{Data availability statement:} Data sharing does not apply to this article as no datasets were generated or analyzed during the current study.



\section*{Acknowledgement}

SB would like to thank the Council of Scientific and Industrial Research (CSIR), India, for financial assistance to carry out this research work [grant no. 09/0874(17164)/ 2023-EMR-I]. SG gratefully acknowledges the financial support for this research work under ARG-MATRICS, grant No: ANRF/ARGM/2025/001570/MTR, Anusandhan National Research Foundation (ANRF), Government of India. RL expresses sincere gratitude for the financial support provided by the Ministry of Education (formerly known as MHRD), Government of India. SB, SG, and RL acknowledge the research facilities available at the Department of Mathematics, NIT Calicut. This work was completed while VK was visiting the Ghent Analysis \& PDE Center at Ghent University. VK gratefully acknowledges the financial support and excellent research facilities provided by the center.



\end{document}